\documentclass{amsart}
\usepackage{amsmath,amsthm,amsfonts, amssymb,amscd}
\usepackage[all]{xy}

\newtheorem{theorem}{Theorem}[section]
\newtheorem{remark}[theorem]{Remark}
\newtheorem{corollary}[theorem]{Corollary}
\newtheorem{proposition}[theorem]{Proposition}

\newcommand{\minusre}{\hspace{0.3em}\raisebox{0.3ex}{\sl \tiny /}\hspace{0.3em}}
\newcommand{\minusli}{\hspace{0.3em}\raisebox{0.3ex}{\sl \tiny $\setminus $}\hspace{0.3em}}
\newcommand{\lex}{\,\overrightarrow{\times}\,}

\newcommand{\lsemiw}{\,\overrightarrow{\ltimes}\,}

\newcommand{\RDP}{\mbox{\rm RDP}}
\newcommand{\RIP}{\mbox{\rm RIP}}
\begin{document}
\title[RDP's and the Lexicographic Product]{Riesz Decomposition Properties and the Lexicographic Product of po-groups}
\date{}%
\maketitle

\begin{center}  \footnote{


This work was supported by  the Slovak Research and Development Agency under contract APVV-0178-11,  grant VEGA No. 2/0059/12 SAV, and GA\v{C}R 15-15286S.
 }
Mathematical Institute,  Slovak Academy of Sciences\\
\v Stef\'anikova 49, SK-814 73 Bratislava, Slovakia\\
$^2$ Depart. Algebra  Geom.,  Palack\'{y} University\\
17. listopadu 12, CZ-771 46 Olomouc, Czech Republic\\

E-mail: {\tt dvurecen@mat.savba.sk}
\end{center}

\begin{abstract}
We establish conditions when a certain type of the Riesz Decomposition Property (RDP) holds in the lexicographic product of two po-groups. It is well known that the resulting product is an $\ell$-group if and only if the first one is linearly ordered and the second one is an $\ell$-group. This can be equivalently studied as po-groups with a special type of the RDP. In the paper we study three different types of RDP's. RDP's of the lexicographic products are important for the study of pseudo effect algebras where infinitesimal elements play an important role both for algebras as well as for the first order logic of valid but not provable formulas.

\end{abstract}

\section{Introduction}

It is well known that there is a very close connection between some algebraic structures and $\ell$-groups or partially ordered groups (= po-groups). A typical case is a result of Mundici, see e.g. \cite{CDM}, saying that every MV-algebra is an interval in an Abelian unital $\ell$-group.  This result was extended for pseudo MV-algebras, a non-commutative generalization of MV-algebras in the sense of \cite{GeIo, Rac}, in \cite{Dvu1}, where a categorical equivalence of the category of pseudo MV-algebras and the category of unital $\ell$-groups (not necessarily Abelian) was established. More about a relation between other algebraic structures like BL-algebras or pseudo BL-algebras and $\ell$-groups was presented in \cite{260}.

There are also partial algebraic structures generalizing MV-algebras or pseudo MV-algebras which are connected with unital po-groups. Typical examples are effect algebras, see \cite{FoBe}, or pseudo effect algebras, see \cite{DvVe1,DvVe2}, with a basic operation $+$, addition, where $a+b$ denotes the disjunction of two mutually excluding events $a$ and $b$. Such structures are important for quantum structures which model events in quantum mechanical measurement, see e.g. \cite{DvPu}. For effect algebras and pseudo effect algebras, a crucial property is the so-called the Riesz Decomposition Property (RDP for short), which denotes a property that every two decompositions of the same element have a joint refinement decomposition. In \cite{Rav}, there was shown that every effect algebra with RDP is an interval in an Abelian unital po-group with RDP. This result was extended in \cite{DvVe1, DvVe2} for pseudo effect algebras with some kind of RDP, called RDP$_1$. In addition, MV-algebras or pseudo MV-algebras can be studied also in the realm of pseudo effect algebras as those with RDP$_2$, another type of RDP.

We recall that not every effect algebra admits RDP. A typical example is the the effect algebra $\mathcal E(H)$, which is the set of Hermitian operators between the zero operator and the identity operator of a complex separable Hilbert space $H$. This algebra is crucial for the study of the so-called Hilbert space quantum mechanics. Nevertheless RDP fails in $\mathcal E(H)$, it can be covered by a family of effect algebras with RDP, for more details see \cite{Pul}. Thus for $\mathcal E(H)$, RDP holds only locally and not globally.

A special class of $\ell$-groups or po-groups consists of groups of the form $G_1\lex G_2$, lexicographic product of two po-groups $G_1$ and $G_2$. The first algebraic model which uses the lexicographic product of $\mathbb Z$, the group of integers, with some Abelian $\ell$-group $G$, is a perfect MV-algebra studied in \cite{DiLe}. In perfect effect MV-algebras every element is either an infinitesimal or a co-infinitesimal.  The logic of perfect pseudo MV-algebras has a counterpart in the  Lindenbaum algebra of the first order \L ukasiewicz logic which is not semisimple, because the valid but unprovable formulas are precisely the formulas that correspond to co-infinitesimal elements of the Lindenbaum algebra, see e.g. \cite{DiGr}. In \cite{DiLe}, there was established that every perfect MV-algebra is an interval in the lexicographic product $\mathbb Z \lex G$, where $G$ is an Abelian $\ell$-group.

An analogous result for perfect effect algebras with RDP was established in \cite{177} and extended in \cite{277}. Lexicographic MV-algebras which are an interval in the lexicographic product $G_1\lex G_2$, where $G_1$ is an Abelian linearly ordered group and $G_2$ is an Abelian $\ell$-group, were studied in \cite{DFL}. The role of the lexicographic product of two po-groups was studied also in \cite{DvKr, DvKo}, where some conditions when the lexicographic product has RDP, RDP$_1$ or RDP$_2$ in special cases of $G_1$ and $G_2$ were established.

As we see, there is a growing interest to algebraic structures which can be represented by the lexicographic product of two po-groups with some type of the RDP's. Therefore, in the present paper we establish conditions when the lexicographic product $G_1\lex G_2$ has some kind of RDP. We present conditions when $G_1\lex G_2$ has RDP$_1$ if $G_1$ is a linearly ordered po-group, and when $G_1 \lex G_2$ has RDP if $G_1$ is an antilattice po-group with RDP, see Section 3. Another group construction closely connected with the lexicographic product is a wreath product and the left and right wreath products. For them we also establish conditions when wreath products have RDP, RDP$_1$ or RDP$_2$, see Sections 4 and 5.

\section{Pseudo Effect Algebras and po-groups}

The main object of our study is a po-group. This means that an algebraic structure $(G;+,-,0,\le)$ is a {\it po-group} (= partially ordered group) if $(G;+,-,0)$ is a group written in an additive way endowed with a partial order $\le$ such that if $a\le b,$ $a,b \in G,$ then $x+a+y \le x+b+y$ for all $x,y \in G.$  We denote by
$G^+:=\{g \in G: g \ge 0\}$ and $G^-:=\{g \in G: g \le 0\}$ the {\it positive cone} and the {\it negative cone} of $G.$ If, in addition, $G$
is a lattice under $\le$, we call it an $\ell$-group (= lattice
ordered group). An element $u \in G^+$ is said to be a {\it strong unit} (or an {\it order unit}) if, given $g \in G,$ there is an integer $n \ge 1$ such that $g \le nu.$ The pair $(G,u),$ where $u$ is a fixed strong unit of $G,$ is said to be a {\it unital po-group}. We recall that  the {\it lexicographic product} of two po-groups $G_1$ and $G_2$, written $G_1 \lex G_2$, is the group $G_1\times G_2$, where the group operations are defined by coordinates, and the lexicographic ordering $\le $ on $G_1 \times G_2$ is defined as follows: For $(g_1,h_1),(g_2,h_2) \in G_1 \times G_2,$  we have $(g_1,h_1)\le (g_2,h_2)$  whenever $g_1 <g_2$ or $g_1=g_2$ and $h_1\le h_2.$


A po-group $G$ is said to be {\it directed} if, given $g_1,g_2 \in G,$ there is an element $g \in G$ such that $g_1,g_2\le g$. Equivalently, $G$ is directed iff every element $g\in G$ can be expressed as a difference of two positive elements of $G.$ For example, every $\ell$-group or every  po-group with strong unit is directed.
For more information on po-groups and $\ell$-groups we recommend the books \cite{Dar, Fuc, Gla}.

A poset $(E;\le)$ is an {\it antilattice} if only comparable elements $a,b\in E$ have a joint  or  meet in $E$. A directed po-group $G$ is an antilattice iff $a\wedge b =0$ implies $a=0$ or $b=0$.

We say that an additively written  po-group $(G;+,-,0,\le)$ satisfies

\begin{enumerate}
\item[(i)]
the {\it Riesz Interpolation Property} (RIP for short) if, for $a_1,a_2, b_1,b_2\in G,$  $a_1,a_2 \le b_1,b_2$  implies there exists an element $c\in G$ such that $a_1,a_2 \le c \le b_1,b_2;$

\item[(ii)]
\RDP$_0$  if, for $a,b,c \in G^+,$ $a \le b+c$, there exist $b_1,c_1 \in G^+,$ such that $b_1\le b,$ $c_1 \le c$ and $a = b_1 +c_1;$

\item[(iii)]
\RDP\  if, for all $a_1,a_2,b_1,b_2 \in G^+$ such that $a_1 + a_2 = b_1+b_2,$ there are four elements $c_{11},c_{12},c_{21},c_{22}\in G^+$ such that $a_1 = c_{11}+c_{12},$ $a_2= c_{21}+c_{22},$ $b_1= c_{11} + c_{21}$ and $b_2= c_{12}+c_{22};$

\item[(iv)]
\RDP$_1$  if, for all $a_1,a_2,b_1,b_2 \in G^+$ such that $a_1 + a_2 = b_1+b_2,$ there are four elements $c_{11},c_{12},c_{21},c_{22}\in G^+$ such that $a_1 = c_{11}+c_{12},$ $a_2= c_{21}+c_{22},$ $b_1= c_{11} + c_{21}$ and $b_2= c_{12}+c_{22}$, and $0\le x\le c_{12}$ and $0\le y \le c_{21}$ imply  $x+y=y+x;$

\item[(v)]
\RDP$_2$  if, for all $a_1,a_2,b_1,b_2 \in G^+$ such that $a_1 + a_2 = b_1+b_2,$ there are four elements $c_{11},c_{12},c_{21},c_{22}\in G^+$ such that $a_1 = c_{11}+c_{12},$ $a_2= c_{21}+c_{22},$ $b_1= c_{11} + c_{21}$ and $b_2= c_{12}+c_{22}$, and $c_{12}\wedge c_{21}=0.$

\end{enumerate}

If, for $a,b \in G^+,$ we have for all $0\le x \le a$ and $0\le y\le b,$ $x+y=y+x,$ we denote this property by $a\, \mbox{\rm \bf com}\, b.$

The RDP will be denoted by the following table:

$$
\begin{matrix}
a_1  &\vline & c_{11} & c_{12}\\
a_{2} &\vline & c_{21} & c_{22}\\
  \hline     &\vline      &b_{1} & b_{2}
\end{matrix}\ \ .
$$

For Abelian po-groups, RDP, RDP$_1,$ RDP$_0$ and RIP are equivalent.

By \cite[Prop 4.2]{DvVe1} for directed po-groups, we have
$$
\RDP_2 \quad \Rightarrow \RDP_1 \quad \Rightarrow \RDP \quad \Rightarrow \RDP_0 \quad \Leftrightarrow \quad  \RIP,
$$
but the converse implications do not hold, in general.  A directed po-group $G$ satisfies \RDP$_2$ iff $G$ is an $\ell$-group, \cite[Prop 4.2(ii)]{DvVe1}.

RDP's are important for the study of algebraic structures like pseudo effect algebras or pseudo MV-algebras  which are a non-commutative generalization of effect algebras and MV-algebras, respectively.

According to \cite{DvVe1, DvVe2}, we say that a {\it pseudo effect algebra} is  a partial algebra  $E=(E; +, 0, 1)$, where $+$ is a partial binary operation and $0$ and $1$ are constants such that, for all $a, b, c
\in E$, the following properties hold

\begin{enumerate}
\item[(i)] $a+b$ and $(a+b)+c$ exist if and only if $b+c$ and
$a+(b+c)$ exist, and in this case $(a+b)+c = a+(b+c)$;

\item[(ii)]
  there is exactly one $d \in E$ and
exactly one $e \in E$ such that $a+d = e+a = 1$;

\item[(iii)]
 if $a+b$ exists, there are elements $d, e
\in E$ such that $a+b = d+a = b+e$;

\item[(iv)] if $1+a$ or $a+1$ exists, then $a = 0$.
\end{enumerate}

If we define $a \le b$ if and only if there exists an element $c\in
E$ such that $a+c =b,$ then $\le$ is a partial ordering on $E$ such
that $0 \le a \le 1$ for any $a \in E.$ It is possible to show that
$a \le b$ if and only if $b = a+c = d+a$ for some $c,d \in E$. We
write $c = a \minusre b$ and $d = b \minusli a.$ Then

$$ (b \minusli a) + a = a + (a \minusre b) = b,
$$
and we write $a^- = 1 \minusli a$ and $a^\sim = a\minusre 1$ for any
$a \in E.$ Then $a^-+a=1=a+a^\sim$ and $a^{-\sim}=a=a^{\sim-}$ for any $a\in E.$

For basic properties of pseudo effect algebras see \cite{DvVe1,DvVe2}. We note that a pseudo effect algebra is an {\it effect algebra}  iff $+$ is commutative, compare with \cite{FoBe}.

A mapping $h$ from a pseudo effect algebra $E$ into another one $F$ is said to be a {\it homomorphism} if (i) $h(1)=1,$ and (ii) if $a+b$ is defined in $E,$ so is defined $h(a)+h(b)$ and $h(a+b)= h(a)+h(b).$ A homomorphism $h$ is an {\it isomorphism} if $h$ is injective and surjective and also $h^{-1}$ is a homomorphism.

Now let  $G$ be a po-group and fix $u \in G^+.$ If we set $\Gamma(G,u):=[0,u]=\{g \in G: 0 \le g \le u\},$ then $\Gamma_E(G,u)=(\Gamma(G,u); +,0,u)$ is a pseudo effect algebra, where $+$ is the restriction of the group addition $+$ to $[0,u],$ i.e. $a+b$ is defined in $\Gamma(G,u)$ for $a,b \in \Gamma(G,u)$ iff $a+b \in \Gamma(G,u).$ Then $a^-=u-a$ and $a^\sim=-a+u$ for any $a \in \Gamma(G,u).$ A pseudo effect algebra which is isomorphic to some $\Gamma_E(G,u)$ for some po-group $G$ with $u>0$ is said to be an {\it interval pseudo effect algebra}.

We say that some type of the Riesz decomposition properties (i)--(v) holds in a pseudo effect algebra $E$ if $G$ or $G^+$ is changed to $E$.

The importance of RDP$_1$ for pseudo effect algebras was established in \cite{DvVe1, DvVe2}:

\begin{theorem}\label{th:2.2}
For every pseudo effect algebra $E=(E;+,0,1)$ with \RDP$_1,$ there is a unique (up to isomorphism of unital po-groups) unital po-group $(G,u)$ with \RDP$_1$\ such that $(E;+,0,1) \cong \Gamma_E(G,u)=(\Gamma(G,u);+,0,u)$.

In addition, $\Gamma_E$ defines a categorical equivalence between the category of pseudo effect algebras with \RDP$_1$ and the category of unital po-groups with \RDP$_1.$
\end{theorem}

The representation theorem for effect algebras with RDP was established in \cite{Rav}. In addition, by \cite{DvVe1,DvVe2}, a pseudo effect algebra has RDP$_2$ iff it is an interval in a unital $\ell$-group.

Another non-commutative structure is a pseudo MV-algebra. According to \cite{GeIo}, a {\it pseudo MV-algebra} is an algebra $ M=(M;
\oplus,^-,^\sim,0,1)$ of type $(2,1,1,$ $0,0)$ such that the
following axioms hold for all $x,y,z \in M$ with an additional
binary operation $\odot$ defined via $$ y \odot x =(x^- \oplus y^-)
^\sim $$
\begin{enumerate}
\item[{\rm (A1)}]  $x \oplus (y \oplus z) = (x \oplus y) \oplus z;$

\item[{\rm (A2)}] $x\oplus 0 = 0 \oplus x = x;$

\item[{\rm (A3)}] $x \oplus 1 = 1 \oplus x = 1;$

\item[{\rm (A4)}] $1^\sim = 0;$ $1^- = 0;$

\item[{\rm (A5)}] $(x^- \oplus y^-)^\sim = (x^\sim \oplus y^\sim)^-;$

\item[{\rm (A6)}] $x \oplus (x^\sim \odot y) = y \oplus (y^\sim
\odot x) = (x \odot y^-) \oplus y = (y \odot x^-) \oplus
x;$

\item[{\rm (A7)}] $x \odot (x^- \oplus y) = (x \oplus y^\sim)
\odot y;$

\item[{\rm (A8)}] $(x^-)^\sim= x.$
\end{enumerate}

A pseudo MV-algebra $M$ is an {\it MV-algebra} iff $a\oplus b = b\oplus a,$ $a,b \in M.$

For example, if $u$ is a strong unit of a (not necessarily Abelian)
$\ell$-group $G$,
$$\Gamma(G,u) := [0,u]
$$
and
\begin{eqnarray*}
x \oplus y &:=&
(x+y) \wedge u,\\
x^- &:=& u - x,\\
x^\sim &:=& -x +u,\\
x\odot y&:= &(x-u+y)\vee 0,
\end{eqnarray*}
then $\Gamma_M(G,u)=(\Gamma(G,u);\oplus, ^-,^\sim,0,u)$ is a pseudo MV-algebra.

The basic result on theory of pseudo MV-algebras \cite{Dvu1} is the following representation theorem.

\begin{theorem}\label{th:2.1}
Every pseudo MV-algebra is an interval $\Gamma(G,u)$ in a unique (up to isomorphism)  unital $\ell$-group $(G,u).$

In addition, the functor $\Gamma_M: (G,u) \mapsto \Gamma_M(G,u)$ defines a categorical equivalence between the variety of pseudo MV-algebras and the category of unital $\ell$-groups.
\end{theorem}

An important note is that every pseudo MV-algebra can be studied also in the realm of pseudo effect algebras with RDP$_2$. Indeed, if $M= (M;\oplus, ^-,^\sim,0,1)$ is a pseudo effect algebra, we define a partial operation $+$ as follows $a+b$ is defined in $M$ iff $a\odot b = 0$ and in such a case, $a+b:=a\oplus b$. Then $(M;+,0,1)$ is a pseudo effect algebra with RDP$_2$, see \cite[Thm 8.3, 8.4]{DvVe2}. Conversely, if $(E; +,0,1)$ is a pseudo effect algebra with RDP$_2,$ then $E$ is a lattice, and by \cite[Thm 8.8]{DvVe2}, $(E; \oplus,^-,^\sim,0,1),$ where
$$
a\oplus b := (b^-\minusli (a\wedge b^-))^\sim,\quad a,b\in E,
$$
is a pseudo MV-algebra.

\section{Riesz Decomposition Properties of Lexicographic Product}

In the section, we establish Riesz decomposition properties of the lexicographic product $G_1\lex G_2$ of two po-groups. We concentrate to two cases: (i) $G_1$ is a linearly ordered po-group and we establish each type of RDP, RDP$_1$,  and RDP$_2$, (ii) $G_1$ is an antilattice and we determine only RDP of the lexicographic product.

We note that the lexicographic product $G_1\lex G_2$, where $G_1$ is a po-group and $G_2$ is a directed po-group, has RDP$_2$ (equivalently, it is an $\ell$-group) iff (i) if $G_2$ is non-trivial, then $G_1$ is linearly ordered and $G_2$ satisfies RDP$_2$ (use  \cite[p. 26 (d)]{Fuc}), or (ii) $G_2=\{e\}$ is trivial, then $G_1 \cong G_1 \lex G_2$, and $G_1$ has to satisfy RDP$_2$.

The following result was proved in \cite[Thm 3.2]{DvKo} for the lexicographic product $G_1\lex G_2$, where $G_1$ is a linearly ordered po-group and $G_2$ is a directed Abelian po-group with RDP. In what follows, we prove it for any type of RDP and for non-Abelian po-groups.

\begin{theorem}\label{th:3.1}
Let $G_1$ be a linearly ordered po-group and  $G_2$ be a directed po-group. Then the lexicographic product $G_1 \lex G_2$ satisfies
\RDP\, {\rm (\RDP$_1$, \RDP$_2$)} if and only if $G_2$ satisfies \RDP\, {\rm (\RDP$_1$, \RDP$_2$)}.
\end{theorem}

\begin{proof}
Let $G_1\lex G_2$ satisfy some type of RDP and let $a_1a_2=b_1b_2$. Then from $(0,a_1)+(0,a_2)=(0,b_1)+(0,b_2)$ we have that $G_2$ has the same type of RDP's.

Now we assume the converse.
If $G_2$ is a trivial group, i.e., $G_2=\{e\}$, then $G_1\lex G_2=G_1\lex \{e\}\cong G_1$ which implies that $G_1\lex G_2$ is linearly ordered, so that it satisfies all RDP, RDP$_1$ and RDP$_2$.

Let us assume that $G$ is non-trivial.
First we prove the case for RDP$_1$, the case for RDP is analogous. We will assume that $G_1=(G_1;+,-,0,\le)$ is written in the additive way and $G_2=(G_2;\cdot,^{-1},e,\le)$ in the multiplicative way, respectively, and the group operation in $G_1\lex G_2$ is assumed in the additive way.

Let $G_2$ be a directed po-group with RDP$_1$.
The positive cone  $(G_1 \lex G_2)^+$ is of the form
$\{(0,a): a \in G_2^+\} \cup \{(n,a): n \in G_1^+\setminus \{0\},\ a \in G_2\}.$
Assume that
$$
(m_1,a_1)+(m_2,a_2) = (n_1,b_1) + (n_2,b_2)$$
holds in $(G_1 \lex G_2)^+.$

(i) Let $(0,a_1)+(0,a_2) = (0,b_1)+ (0,b_2).$  Then $a_1,a_2,b_1,b_2 \in G_2^+$ and RDP$_1$ for this case  follows from RDP$_1$ for $G_2.$

(ii) $(0,a_1) + (n,a_2 )=
(0,  b_1 ) + (n,  b_2 )$ for $a_1,b_1 \ge e$, $a_2,b_2 \in G_2$ for each  $n \in G_1^+\setminus\{0\}$. Then $a_1a_2=b_1b_2$. While $G_2$ is directed,  there is an element $d \in G_2$ such that $a_2,b_2 \ge d$. Then $a_1a_2d^{-1}=b_1b_2d^{-1}$ and for them we have the RDP$_1$ decomposition

$$
\begin{matrix}
a_1 &\vline & c_{11} & c_{12}\\
a_2d^{-1} &\vline & c_{21} & c_{22}\\
  \hline     &\vline      &b_1 & b_2d^{-1}
\end{matrix}\ \ ,
$$
where $c_{12}\, \mbox{\rm \bf com}\, c_{21}$. Then

$$
\begin{matrix}
a_1  &\vline & c_{11} & c_{12}\\
a_2 &\vline & c_{21} & c_{22}d\\
  \hline     &\vline      &b_1 & b_2
\end{matrix}\ \
$$
and
$$
\begin{matrix}
(0, a_1)  &\vline & (0, c_{11}) & (0, c_{12})\\
(n, a_2) &\vline & (0, c_{21}) & (n, c_{22}d)\\
  \hline     &\vline      &(0, b_1) & (n, b_2)
\end{matrix}\ \
$$
is an RDP$_1$ decomposition for (ii) in the po-group $G_1\lex G_2$.

(iii) $(n,   a_1 ) + (0,  a_2 )=
(n,  b_1 ) + (0,  b_2 )$ for $a_2,b_2 \ge e$, $a_1,b_1 \in G_2$ for  $n \in G_1^+\setminus \{0\}$. The directness of $G_2$ implies, there is $d \in G_2$ such that $d\le a_1,a_2,b_1,b_2$. Equality (iii) can be rewritten in the equivalent form $(n,   d^{-1}a_1 ) + (0,  a_2d^{-1} )=
(n,  d^{-1}b_1 ) + (0,  b_2d^{-1} )$ which yields $d^{-1}a_1a_2d^{-1}=d^{-1}b_1b_2d^{-1}$.
It entails an RDP$_1$ decomposition in the po-group $G_2$

$$
\begin{matrix}
d^{-1}a_1 &\vline & c_{11} & c_{12}\\
a_2d^{-1} &\vline & c_{21} & c_{22}\\
  \hline     &\vline      &d^{-1}b_1 & b_2
d^{-1}\end{matrix}\ \ ,
$$
consequently,
$$
\begin{matrix}
a_1  &\vline & dc_{11} & c_{12}\\
a_2 &\vline & c_{21} & c_{22}d\\
  \hline     &\vline      &b_1 & b_2
\end{matrix}\ \ ,
$$
and it gives an RDP$_1$ decomposition of (iii) in the po-group $G_1\lex G_2$

$$
\begin{matrix}
(n, a_1)  &\vline & (n, dc_{11}) & (0, c_{12})\\
(0, b_1) &\vline & (0, c_{21}) & (0, c_{22}d)\\
  \hline     &\vline      &(n, b_1) & (0, b_2 )
\end{matrix}\ \ .
$$

(iv) $(n,   a_1 ) + (0,  a_2 )=
(0,  b_1) + (n,  b_2 )$ for $a_1,b_2\in G_2 $, $a_2,b_1 \ge e$ for  $n \in G_1^+\setminus\{0\}$.

Then $a_1a_2 = b_1b_2$  which implies $b_1^{-1}a_1= b_2 a_2^{-1}$.   If we use the decomposition

$$
\begin{matrix}
(n, a_1)  &\vline & (0, b_1) & (n, b_1^{-1}a_1)\\
(0, a_2) &\vline & (0,e) & (0, a_2)\\
  \hline     &\vline      &(0, b_1) & (n, b_2 )
\end{matrix}\ \ ,
$$
we see that it gives an RDP$_1$ decomposition for (iv); trivially  $(0,e)\, \mbox{\rm \bf com}\, (n, b_1^{-1}a_1)$.

(v) $(n,   a_1) + (0,  a_2 )=
(m_1,  b_1 ) + (m_2,  b_2 )$ for $a_1,b_1, b_2\in G_2 $, $a_2\ge e$, where $m_1,m_2\in G_1^+\setminus\{0\}$ and $m_1+m_2=n$. Then $a_1a_2=b_1b_2$. Hence, the following table gives an RDP$_1$ decomposition for (v)

$$
\begin{matrix}
(n, a_1)  &\vline & (m_1, b_1) & (m_2, b^{-1}_1a_1)\\
(0, a_2) &\vline & (0,e) & (0, a_2)\\
  \hline     &\vline      &(m_1, b_1) & (m_2, b_2)
\end{matrix}\ \ .
$$

(vi) $(0,   a_1) + (n,  a_2)=
(m_1,  b_1) + (m_2,  b_2 )$ for $a_2,b_1,b_2\in G_2 $, $a_1\ge e$, where $m_1,m_2\in G_1^+\setminus\{0\}$ and $m_1+m_2=n$. Then we have $b_1b_2=a_1a_2$ and  the following RDP$_1$ decomposition

$$
\begin{matrix}
(0, a_1)  &\vline & (0, a_1) & (0,e)\\
(n, a_2) &\vline & (m_1, a_1^{-1}b_1) & (m_2, b_2)\\
  \hline     &\vline      &(m_1, b_1) & (m_2, b_2)
\end{matrix}\ \ .
$$

(vii) $(n_1,   a_1 ) + (n_2,  a_2 )=
(m_1,  b_1) + (m_2,  b_2 )$ for $a_1,a_2,b_1,b_2\in G_2$, where $n_1, n_2,$ $ m_1,m_2 \in G_1^+\setminus\{0\}$, $n_1+n_2=n=m_1+m_2$ and $m_1> n_1$. Then $a_1a_2=b_1b_2$, and since $-n_1+n=n_2$, (vii)  has the following RDP$_1$ decomposition

$$
\begin{matrix}
(m_1, b_1)  &\vline & (n_1, a_1) & (-n_1+m_1, a_1^{-1}b_1)\\
(m_2, b_2 ) &\vline & (0,e) & (m_2, b_2)\\
  \hline     &\vline      &(n_1, a_1) & (n_2, a_2)
\end{matrix}\ \ \mbox{ if } m_1 >n_1
$$
is an RDP$_1$ decomposition.

(viii) $(n_1,   a_1) + (n_2,  a_2 )=
(m_1,  b_1 ) + (m_2,  b_2 )$ for $a_1,a_2,b_1,b_2 \in G_2 $,  where $n_1, n_2, m_1,m_2 \in G_1^+\setminus\{0\}$, $n_1+n_2=n=m_1+m_2$ and $n_1> m_1$. Then $a_1a_2=b_1b_2$. Hence, the following table

$$
\begin{matrix}
(n_1, a_1)  &\vline & (m_1, b_1) & (-m_1+n_1, b_1^{-1}a_1)\\
(n_2, a_2) &\vline & (0,e) & (n_2, a_2)\\
  \hline     &\vline      &(m_1, b_1) & (m_2, b_2)
\end{matrix}\ \ \mbox{ if } n_1 >m_1
$$
gives an RDP$_1$ decomposition for (viii). By a way, (viii) follows also from (vii) when we rewrite (viii) in the equivalent form $(m_1,  b_1 ) + (m_2,  b_2 )=(n_1,a_1) + (n_2,a_2 )$.

(ix) $(n_1,   a_1 ) + (n_2,  a_2)=
(m_1,  b_1) + (m_2,  b_2 )$ for $a_1,a_2,b_1,b_2\in G_2 $, where $n_1, n_2, m_1,$ $ m_2 \in G_1^+\setminus\{0\}$, $n_1+n_2=n=m_1+m_2$ and $n_1= m_1$.  Then $a_1a_2= b_1b_2$. The directness of $G_2$ entails that there is $d\in G_2$ such that $a_1,a_2,b_1,b_2\ge d$. Hence, $d^{-1}a_1a_2d^{-1}= d^{-1}b_1b_2d^{-1}$. The RDP$_1$ holding in $G_2$ entails the following RDP$_1$ table

$$
\begin{matrix}
d^{-1}a_1  &\vline & c_{11} & c_{12}\\
a_2d^{-1} &\vline & c_{21} & c_{22}\\
  \hline     &\vline      &d^{-1}b_1 & b_2d^{-1}
\end{matrix}\ \ ,
$$
so that

$$
\begin{matrix}
a_1  &\vline & dc_{11} & c_{12}\\
a_2 &\vline & c_{21} & c_{22}d\\
  \hline     &\vline      &b_1 & b_2
\end{matrix}\ \ .
$$

It gives an RDP$_1$ decomposition of (ix)

$$
\begin{matrix}
(n_1, a_1 )  &\vline & (n_1, dc_{11}) & (0, c_{12} )\\
(n_2,  a_2) &\vline & (0, c_{21}) & (n_2, c_{22}d)\\
  \hline     &\vline      &(n_1, b_1)& (n_2, b_2 )
\end{matrix}\ \ .
$$

Now let $G_2$ satisfy RDP$_2$.  By \cite[Prop 4.2(ii)]{DvVe1}, a directed po-group $G_2$ satisfies \RDP$_2$ iff $G_2$ is an $\ell$-group. Then by \cite[p. 26 (d)]{Fuc}, the lexicographic product is an $\ell$-group, so that again by \cite[Prop 4.2(ii)]{DvVe1}, the lexicographic product $G_1 \lex G_2$ satisfies RDP$_2$.

Another proof of this implication follows all previous steps (i)--(ix)  for RDP$_2$ assumptions which prove that the po-group $G_1 \lex G_2$ has RDP$_2$.
\end{proof}

\begin{remark}\label{re:3.2}
{\rm We note that in \cite[Cor 2.12]{Goo}, there are presented conditions when the lexicographic product of two Abelian po-groups has the interpolation property, which means when the lexicographic product satisfies RDP, equivalently, RDP$_1$. By Theorem \ref{th:3.1}, we know conditions for non-Abelian po-groups only if $G_1$ is linearly ordered, and the proof of cases (vii)--(ix) strongly used this fact. Therefore, we suggest to find a proof of Theorem \ref{th:3.1} without the assumption that $G_1$ is linearly ordered, simply assuming that $G_1$ satisfies RDP or RDP$_1$, respectively, or to find a counterexample.  We underline that the proof of the validity of RDP$_2$ has assumed $G_1$ is linearly ordered as it follows from the last line of the proof of the latter theorem.}
\end{remark}

A partial answer to the problem from Remark \ref{re:3.2} is the following result concerning RDP.

\begin{theorem}\label{th:3.3}
Let $G_1$ be an antilattice po-group and $G_2$ be a directed po-group. Then $G_1\lex G_2$ satisfies \RDP\, if and only if both $G_1$ and $G_2$ satisfy \RDP.
\end{theorem}

\begin{proof}
The implication ``if $G_1\lex G_2$ has RDP, then both $G_1$ and $G_2$ have RDP" follows from the beginning of the proof od Theorem \ref{th:3.1}.

The converse implication. We will assume that $(G_1;+,-,0,\le)$ is written in an additive way, $(G_2;\cdot,^{-1},e,\le)$ is in a multiplicative way and the lexicographic product $G_1 \lex G_2$ is written in an additive way.

The steps (i)--(vi) from the proof of Theorem \ref{th:3.1} remain.
It is necessary to exhibit the case $(n_1,   a_1 ) + (n_2,  a_2 )=
(m_1,  b_1) + (m_2,  b_2 )$ for $a_1,a_2,b_1,b_2\in G_2$, where $n_1, n_2,$ $ m_1,m_2 \in G_1^+\setminus\{0\}$, $n_1+n_2=n=m_1+m_2$. Then $a_1a_2=b_1b_2$, and $-n_1+n=n_2$.

If $n_1$ and $m_2$ are comparable, we use one of the cases (vii)--(ix) of Theorem \ref{th:3.1} and we are ready. Thus we suppose that $n_1$ and $m_1$ are not comparable, consequently, $n_2$ and $m_2$ are also not comparable. Since there exists no greatest lower bound of $n_1$ and $m_1$, there is $n_0\in G_1$ such that $0<n_0<n_1,m_1$. Due to the same reason, there is $m_0\in G_1$ such that $0<m_0<n_2,m_2$. Hence, we have $(-n_0+n_1)+(n_2-m_0) = (-n_0+m_1)+(m_2-m_0)$ where all elements in brackets are strictly positive. Using RDP for $G_1$, we have the decomposition

$$
\begin{matrix}
-n_0+n_1  &\vline & n_{11} & n_{12}\\
n_2-m_0 &\vline & n_{21} & n_{22}\\
  \hline     &\vline      &-n_0+m_1 & m_2-m_0
\end{matrix}\ \
$$
which yields

$$
\begin{matrix}
n_1  &\vline & n_0+n_{11} & n_{12}\\
n_2 &\vline & n_{21} & n_{22}+m_0\\
  \hline     &\vline      &m_1 & m_2-m_0
\end{matrix}\ \ ,
$$
where the elements in the upper left-side corner and in the lower right-side corner are strictly positive.

Therefore, we can always find an RDP decomposition table

$$
\begin{matrix}
n_1  &\vline & n_{11} & n_{12}\\
n_2 &\vline & n_{21} & n_{22}\\
  \hline     &\vline      &m_1 & m_2
\end{matrix}\ \ ,
$$
where $n_{11}>0$ and $n_{22}>0$. In addition, if $n_{12}=0$, then $n_1=n_{11}\le m_1$ which is impossible, and also $n_{12}>0$. Similarly, $n_{21}>0$.

From the directness of $G_2$, there is $d\le a_1,a_2,b_1,b_2$, such that from $d^{-1}a_1a_2d^{-1}=d^{-1}b_1b_2d^{-1}$ we have the RDP table in $G_2$

$$
\begin{matrix}
a_1  &\vline & dc_{11} & c_{12}\\
a_2 &\vline & c_{21} & c_{22}d\\
  \hline     &\vline      &b_1 & b_2
\end{matrix}\ \ .
$$

Therefore, we have a final RDP-decomposition for our case

$$
\begin{matrix}
(n_1,a_1)  &\vline & (n_{11},dc_{11}) & (n_{12},c_{12})\\
(n_2,a_2) &\vline & (n_{21},c_{21}) & (n_{22},c_{22}d) \\
  \hline     &\vline      &(m_1,b_1) & (m_2,b_2)
\end{matrix}\ \ .
$$
\end{proof}

\section{Wreath Product and RDP's}

In the section we establish Riesz decomposition properties of the wreath product. Such a kind of product was used for study of $n$-perfect kite pseudo effect algebras in  \cite{BoDv}.

Now let $(A;\cdot,^{-1},0,\le)$ be a linearly ordered po-group and $(G;\cdot,^{-1},e,\le)$ be a po-group. Let $W=A \lsemiw G^A$ be the semidirect product, i.e., the elements of $W$ are of the form $(n,\langle g_a:a \in A\rangle)$, where $n\in A$, $g_a \in G$ for each $a \in A$. The multiplication $*$ on $W$ is defined as follows

$$(n,\langle g_a: a \in A\rangle) * (m,\langle h_a: a \in A\rangle):=
(nm, \langle g_ah_{an}: a \in A\rangle),\eqno(4.1)
$$
with the neutral element $(0,e^A)$ where $e^A=\langle e_a: e_a=e, a\in A\rangle$, the inverse of an element $(n,\langle g_a: a \in A\rangle)$ is the element
$$ (n,\langle g_a: a \in A\rangle)^{-1}:=(n^{-1}, \langle g^{-1}_{an^{-1}}: a \in A\rangle),
$$
and the ordering is defined by $(n,\langle g_a: a \in A\rangle)\le (m,\langle h_a: a \in A\rangle)$ iff $n<m$ or $n=m$ and $g_a\le h_a$ for each $a \in A$.
Then $(W; *,^{-1},(0,e^A),\le)$ is a po-group called the ({\it unrestricted}) {\it Wreath product} of $A$ and $G$; we write also $W=A \mbox{ Wr }G$, see \cite[Ex 1.3.27]{Gla}. If $G$ is an $\ell$-group, so is the Wreath product. If $G$ is trivial, then the Wreath product is isomorphic to $A$.

The subgroup of all elements $(n,\langle g_a: a \in A\rangle)$, where $g_a = e$ for all but a finite number of $a \in A$, is said to be the {\it restricted} {\it Wreath product} of $A$ and $G$; we write $A \mbox { wr } G$. It is a po-group. If $G$ is a trivial po-group, then the restricted Wreath product is isomorphic to $A$.

Besides these restricted Wreath products we introduce according to \cite[Ex 1.3.28]{Gla} two its special kinds: The subgroup of $W=\mathbb Z \ltimes G^\mathbb Z$ consisting of all elements of the form $(n,\langle g_i: i \in \mathbb Z\rangle)$, where $g_i=e$ for all but finitely many $i \in \mathbb Z$, can be ordered in two ways: an element $(n,\langle g_i: i\in \mathbb Z\rangle) \in W$ is positive if either $n>0$ or $n=0$ and $g_j>e$ where $j$ is the greatest integer $i$ such that $g_i\ne e$. We call $W$ the {\it right wreath product} of $\mathbb Z$ and $G$, and we write$ W=\mathbb Z\, \overrightarrow{\mbox{wr}}\, G$.

The second ordering of $W$ is as follows: An element $(n,\langle g_i: i\in \mathbb Z\rangle)\in W$ is positive iff either $n>0$ or $n=0$ and $g_j>e$ where $j$ is the least integer $i$ such that $g_i\ne e$. We call $W$ the {\it left wreath product} of $\mathbb Z$, and we write $W=\mathbb Z\, \overleftarrow{\mbox{wr}}\, G$.

Both products are po-groups, and if $G$ is a linearly ordered group so are both ones. We note that according to \cite[Cor 61.17]{Dar}, in such a case both products generate different varieties of $\ell$-groups. If, in addition, $G$ is a trivial po-group, then both wreath products are isomorphic to the po-group $\mathbb Z$.

Similarly as for the lexicographic product we exhibit when the Wreath product has RDP or RDP$_1$. We note that for its proof we use the methods inspired by the proof of Theorem \ref{th:3.1}.

\begin{theorem}\label{th:4.1}
Let $A$ be a linearly ordered po-group and $G$ be a directed po-group. Then the Wreath product $A\, \mbox {\rm Wr } G$ satisfies \RDP\, (\RDP$_1$, \RDP$_2$, respectively) if and only if $G$ satisfies \RDP\, (\RDP$_1$, \RDP$_2$, respectively).
\end{theorem}

\begin{proof}
If $G$ is trivial, then $A\lsemiw G^A \cong A$, and it satisfies RDP, RDP$_1$ and RDP$_2$, because $A$ is linearly ordered.

Thus let $G$ be non-trivial.
The positive cone of the Wreath product $A\lsemiw G^A$ is the set $\{(e,\langle g_a: a \in A\rangle): g_a \ge e$ for each $a\in A\}\cup \bigcup_{n>1} \{(n, \langle g_a: a \in A\rangle), \ g_a \in G \mbox{ for each } a \in A\}$.

In what follows, we will deal only with the case RDP$_1$; the case of RDP is similar.

Let the Wreath product $A\lsemiw G^A$ satisfy RDP$_1$ and assume that $g_1g_2=h_1h_2$ for $g_1,g_2,h_1,h_2 \in G^+$. Fix an element $a_0\in A$ and for $j=1,2,$ let us define $A_j=(0,\langle f_a^j \colon a\in A\rangle)$ by $f_a^j=a_j$ if $a=a_0$ and $f_a^j=e$ otherwise, $B_j= (0,\langle g_a^j \colon a\in A\rangle)$ by $g_a^j=b_j$ if $a=a_0$ and $g_a^j=e$ otherwise. Then $A_1 * A_2=B_1* B_2$ so that there are $E_{11}=(0,\langle e^{11}_a \colon a\in A\rangle),$ $E_{12}=(0,\langle e^{12}_a \colon a\in A\rangle),$ $E_{21}=(0,\langle e^{21}_a \colon a\in A\rangle),$ and $E_{22}=(0,\langle e^{22}_a \colon a\in A\rangle),$ such that $A_1=E_{11}* E_{12},$ $A_2=E_{21}* E_{22},$ $B_1=E_{11}* E_{21},$ and $B_2=E_{12}* E_{22}.$ Hence, we see that $G$ satisfies RDP$_1$.

Conversely, let $G$ satisfy RDP$_1$.  To be more compact, we will write
$(n,\langle g_a\rangle)$ instead of $(n,\langle g_a: a \in A\rangle)$.  We have the following cases.

(i) $(0, \langle  g_a \rangle) *(0,\langle  h_a\rangle)=
(0,\langle  u_a \rangle) *(0,\langle  v_a \rangle)$ for $g_a,h_a,u_a,v_a \ge e$ for each $i \in I$. The proof of this case trivial.

(ii) $(0, \langle  g_a \rangle) *(n,\langle  h_a \rangle)=
(0,\langle  u_a \rangle) *(n,\langle  v_a \rangle)$ for $g_a,u_a \ge e$, $h_a,v_a \in G$ for each $a \in A$, and $n>0 $. Then $g_ah_a=u_av_a$ for each $a \in A$. Since $G$ is directed, for any $a\in A$, there is an element $d_a \in G$ such that $h_a,v_a \ge d_a$. Then $g_ah_ad_a^{-1}=u_av_ad_a^{-1}$ and for them we have the RDP$_1$ decomposition

$$
\begin{matrix}
g_a  &\vline & c_a^{11} & c_a^{12}\\
h_ad_a^{-1} &\vline & c_a^{21} & c_a^{22}\\
  \hline     &\vline      &u_a & v_ad_a^{-1}
\end{matrix}\ \ ,
$$
where $c_a^{12}\, \mbox{\rm \bf com}\, c_a^{21}$. Then

$$
\begin{matrix}
g_a  &\vline & c_a^{11} & c_a^{12}\\
h_a &\vline & c_a^{21} & c_a^{22}d_a\\
  \hline     &\vline      &u_a & v_a
\end{matrix}\ \
$$
and
$$
\begin{matrix}
(0,\langle g_a\rangle)  &\vline & (0,\langle c_a^{11}\rangle) & (0,\langle c_a^{12}\rangle)\\
(n,\langle h_a\rangle) &\vline & (0,\langle c_a^{21}\rangle) & (n,\langle c_a^{22}d_a\rangle)\\
  \hline     &\vline      &(0,\langle u_a\rangle) & (n,\langle v_a \rangle)
\end{matrix}\ \
$$
is an RDP$_1$ decomposition for (ii) in the po-group $A \lsemiw G^A$.

(iii) $(n, \langle  g_a \rangle) *(0,\langle  h_a \rangle)=
(n,\langle  u_a \rangle) *(0,\langle  v_a \rangle)$ for $h_a,v_a \ge e$, $g_a,u_a \in G$ for each $a \in A$, and $n >0$. The directness of $G$ implies, for each $a \in A$, there is $d_a \in G$ such that $d_a\le g_a,h_a,u_a,v_a$. Equality (iii) can be rewritten in the equivalent form $(n, \langle  d_a^{-1}g_a \rangle) *(0,\langle  h_ad_a^{-1} \rangle)=
(n,\langle  d_a^{-1}u_a \rangle) *(0,\langle  v_ad_a^{-1} \rangle)$ which yields $d_a^{-1}g_ah_{an}d^{-1}_{an}=d_a^{-1}u_av_{an}d_{an}^{-1}$.
It entails an RDP$_1$ decomposition in the po-group $G$

$$
\begin{matrix}
d_a^{-1}g_a  &\vline & c_a^{11} & c_a^{12}\\
h_{an}d^{-1}_{an} &\vline & c_a^{21} & c_a^{22}\\
  \hline     &\vline      &d_a^{-1}u_a & v_{an}
d^{-1}_{an}\end{matrix}\ \ ,
$$
consequently,
$$
\begin{matrix}
g_a  &\vline & d_ac_a^{11} & c_a^{12}\\
h_{an} &\vline & c_a^{21} & c_a^{22}d_{an}\\
  \hline     &\vline      &u_a & v_{an}
\end{matrix}\ \ ,
$$
and it gives an RDP$_1$ decomposition of (iii) in the Wreath product $A\lsemiw G^I$

$$
\begin{matrix}
(n,\langle g_a\rangle)  &\vline & (n,\langle d_ac_a^{11}\rangle) & (0,\langle c_{an^{-1}}^{12}\rangle)\\
(0,\langle h_a\rangle) &\vline & (0,\langle c_{an^{-1}}^{21}\rangle) & (0,\langle c_{an^{-1}}^{22}d_a\rangle)\\
  \hline     &\vline      &(n,\langle u_a\rangle) & (0,\langle v_a \rangle)
\end{matrix}\ \ .
$$

(iv) $(n, \langle  g_a \rangle) *(0,\langle  h_a \rangle)=
(0,\langle  u_a \rangle) *(n,\langle  v_a \rangle)$ for $g_a,v_a\in G $, $h_a,u_a \ge e$ for each $a \in A$, $n >0$.

Then $g_ah_{an} = u_av_a$ for each $a \in A$, which implies $u_a^{-1}g_a= v_a h^{-1}_{an}$.   If we use the decomposition

$$
\begin{matrix}
(n,\langle g_a\rangle)  &\vline & (0,\langle u_a\rangle) & (n,\langle u_a^{-1}g_a\rangle)\\
(0,\langle h_a\rangle) &\vline & (0,e^A) & (0,\langle h_a\rangle)\\
  \hline     &\vline      &(0,\langle u_a\rangle) & (n,\langle v_a \rangle)
\end{matrix}\ \ ,
$$
then it gets an RDP$_1$ decomposition for (iv); trivially  $(0,e^A)\, \mbox{\rm \bf com}\, (n,\langle u_a^{-1}g_a\rangle)$.

(v) $(n, \langle  g_a \rangle) *(0,\langle  h_a \rangle)=
(m_1,\langle  u_a \rangle) *(m_2,\langle  v_a \rangle)$ for $g_a,u_a, v_a\in G $, $h_a\ge e$ for each $a \in A$, where $m_1,m_2>0$ and $m_1m_2=n$. Then $g_ah_{an}=u_av_{am_1}$ for each $a\in A$. Hence, the following table gives an RDP$_1$ decomposition for (v)

$$
\begin{matrix}
(n,\langle g_a\rangle)  &\vline & (m_1,\langle u_a\rangle) & (m_2,\langle u^{-1}_{am_1^{-1}}g_{am_1^{-1}}\rangle)\\
(0,\langle h_a\rangle) &\vline & (0,e^A) & (0,\langle h_a\rangle)\\
  \hline     &\vline      &(m_1,\langle u_a\rangle) & (m_2,\langle v_a \rangle)
\end{matrix}\ \ .
$$

(vi) $(0, \langle  g_a \rangle) *(n,\langle  h_a \rangle)=
(m_1,\langle  u_a \rangle) *(m_2,\langle  v_a \rangle)$ for $h_a,u_a, v_a\in G $, $g_a\ge e$ for each $a \in A$, where $m_1,m_2>0$ and $m_1m_2=n$. For it we have $u_av_{am_1}=g_ah_a$ for each $a \in A$ and  the following RDP$_1$ decomposition

$$
\begin{matrix}
(0,\langle g_a\rangle)  &\vline & (0,\langle g_a\rangle) & (0,e^A)\\
(n,\langle h_a\rangle) &\vline  &(m_1,\langle g_a^{-1}u_a\rangle) & (m_2,\langle v_a\rangle)\\
  \hline     &\vline      &(m_1,\langle u_a\rangle) & (m_2,\langle v_a \rangle)
\end{matrix}\ \ .
$$

(vii) $(n_1, \langle  g_a \rangle) *(n_2,\langle  h_a \rangle)=
(m_1,\langle  u_a \rangle) *(m_2,\langle  v_a \rangle)$ for $g_a, h_a,u_a, v_a\in G $, for each $a \in A$, where $n_1, n_2, m_1,m_2 >0$, $n_1n_2=n=m_1m_2$ and $m_1> n_1$. Then $g_ah_{an_1}=u_av_{a{m_1}}$ for each $a \in A$, and using the equality $n^{-1}_1=n_2n^{-1}$, case (vii)  has the following RDP$_1$ decomposition

$$
\begin{matrix}
(m_1,\langle u_a\rangle)  &\vline & (n_1,\langle g_a\rangle) & (n_1^{-1}m_1,\langle g_{a{n_1}^{-1}}^{-1}u_{an_1^{-1}}\rangle)\\
(m_2,\langle v_a \rangle) &\vline & (0,e^A) & (m_2,\langle v_a\rangle)\\
  \hline     &\vline      &(n_1,\langle g_a\rangle) & (n_2,\langle h_a\rangle)
\end{matrix}\ \ \mbox{ if } m_1 >n_1
$$
is an RDP$_1$ decomposition.

(viii) $(n_1, \langle  g_a \rangle) *(n_2,\langle  h_a \rangle)=
(m_1,\langle  u_a \rangle) *(m_2,\langle  v_a \rangle)$ for $g_a, h_a,u_a, v_a\in G $, for each $a \in A$, where $n_1, n_2, m_1,m_2 >0$, $n_1n_2=n=m_1m_2$ and $n_1> m_1$. Then $g_a h_{an_1}=u_av_{am_1}$ for each $a \in A$. Hence, the following table

$$
\begin{matrix}
(n_1,\langle g_a\rangle)  &\vline & (m_1,\langle u_a\rangle) & (m_1^{-1}n_1,\langle u_{am_1^{-1}}^{-1}g_{am_1^{-1}}\rangle)\\
(n_2,\langle h_a \rangle) &\vline & (0,e^A) & (n_2,\langle h_a\rangle)\\
  \hline     &\vline      &(m_1,\langle u_a\rangle) & (m_2,\langle v_a\rangle)
\end{matrix}\ \ \mbox{ if } n_1 >m_1
$$
gives an RDP$_1$ decomposition for (viii).

(ix) $(n_1, \langle  g_a \rangle) *(n_2,\langle  h_a \rangle)=
(m_1,\langle  u_a \rangle) *(m_2,\langle  v_a \rangle)$ for $g_a, h_a,u_a, v_a\in G $, for each $a \in A$, where $n_1, n_2, m_1,m_2 >0$, $n_1n_2=n=m_1m_2$ and $n_1= m_1$.  Then $g_ah_{an_1}= u_av_{an_1}$. The directness of $G$ entails that, for
every $a \in A$, there is $d_a\in G$ such that $g_a,h_a,u_a,v_a\ge d_a$. Hence, $d_a^{-1}g_ah_{an_1}d^{-1}_{an_1}= d_a^{-1}u_av_{an_1}d^{-1}_{an_1}$. The RDP$_1$ holding in $G$ gives the following RDP$_1$ table

$$
\begin{matrix}
d_a^{-1}g_a  &\vline & c_a^{11} & c_a^{12}\\
h_{an_1}d^{-1}_{an_1} &\vline & c_a^{21} & c_a^{22}\\
  \hline     &\vline      &d_a^{-1}u_a & v_{an_1}d^{-1}_{an_1}
\end{matrix}\ \ ,
$$
so that

$$
\begin{matrix}
g_a  &\vline & d_ac_a^{11} & c_a^{12}\\
h_{an_1} &\vline & c_a^{21} & c_a^{22}d_{an_1}\\
  \hline     &\vline      &u_a & v_{an_1}
\end{matrix}\ \ .
$$

It gives the RDP$_1$ decomposition of (ix)

$$
\begin{matrix}
(n_1,\langle g_a \rangle)  &\vline & (n_1,\langle d_ac_a^{11}\rangle) & (0,\langle c_{an_1^{-1}}^{12}\rangle )\\
(n_2,\langle  h_a\rangle) &\vline & (0,\langle c_{an_1^{-1}}^{21}\rangle) & (n_2,\langle c_{an_1^{-1}}^{22}d_{an_1^{-1}}\rangle)\\
  \hline     &\vline      &(n_1,\langle u_a \rangle)& (n_2,\langle v_a \rangle)
\end{matrix}\ \ .
$$

Now assume that $G_2$ satisfy RDP$_2$.  By \cite[Prop 4.2(ii)]{DvVe1}, a directed po-group $G$ satisfies \RDP$_2$ iff $G$ is an $\ell$-group. It is easy to verify that if $G$ is an $\ell$-group, so is $A\lsemiw G^A$.
\end{proof}

\begin{corollary}\label{co:4.2}
Let $A$ be a linearly ordered po-group and $G$ be a directed po-group. The restricted Wreath product $A\,\mbox{\rm  wr }G$ satisfies  \RDP\, (\RDP$_1$, \RDP$_2$, respectively) if and only if $G$ satisfies \RDP\, (\RDP$_1$, \RDP$_2$, respectively).
\end{corollary}

\begin{proof}
It is identical to the proof of Theorem \ref{th:4.1}.
\end{proof}

\section{Right and Left Wreath Products and RDP's}

In this section we exhibit conditions when a kind of RDP holds in the right and left wreath products.

We note that for $a=(n,\langle x_i: i \in \mathbb Z\rangle)$ and $b=(m,\langle y_i: i \in \mathbb Z \rangle)$ from the right wreath product we have
$a\le b$ iff (i) $n< m$ or (ii) $n=m$ and $x_{i_0} < y_{i_0}$, where $i_0:=\max\{i\in \mathbb Z: x_i \ne y_i\}$. Dually for $a,b$ from the left wreath product. That is, $a\le b$ iff (i) $n<m$ or (ii) $x_{i_0}<y_{i_0}$, where $i_0:=\min \{i \in \mathbb Z: x_i \ne y_i\}$.

\begin{proposition}\label{pr:5.1}
Let $G$ be a directed po-group. The following statements are equivalent
\begin{enumerate}
\item[{\rm (i)}] $\mathbb Z\, \overrightarrow{\mbox{\rm wr}}\, G$ satisfies \RDP$_2$.
\item[{\rm (ii)}] $\mathbb Z\, \overleftarrow{\mbox{\rm wr}}\, G$ satisfies \RDP$_2$.
\item[{\rm (iii)}] $G$ is a linearly ordered group.
\end{enumerate}
In  any such case, the right and left wreath products are linearly ordered groups.
\end{proposition}

\begin{proof}
(i) $\Rightarrow$ (ii),(iii). If $G$ is linearly ordered, by \cite[Ex 1.3.28]{Gla}, the right and left wreath products are linearly ordered groups.  By \cite[Prop 4.2(ii)]{DvVe1}, a directed po-group  satisfies \RDP$_2$ iff $G$ is an $\ell$-group, so that the both wreath products satisfy RDP$_2$.

(ii) $\Rightarrow$ (i). If $G$ is trivial, i.e. $G=\{e\}$, the statement is trivial. Thus we assume that $G$ is non-trivial and let $\mathbb Z\, \overrightarrow{\mbox{\rm wr}}\, G$ satisfy RDP$_2$. Then the right wreath product is an $\ell$-group. We claim that $G$ is linearly ordered. If not, let $a,b>e$ be two non-comparable elements such that there is an element $c \in G$ with $e<c<a,b$. Define two elements $x=(0,\langle x_i: i \in \mathbb Z\rangle)$ and $y=(0,\langle y_i: i \in \mathbb Z\rangle)$, where $x_0=a$, $y_0=b$ and $x_i=y_i=e$ for each $i \in \mathbb Z \setminus \{0\}$. Let $z =(0,\langle z_i: i \in \mathbb Z\rangle) =x\wedge y \in \mathbb Z\, \overrightarrow{\mbox{\rm wr}}\, G$. We assert that every $z_i=e$ for each integer $i>0$. If not, take the largest index $i_0>0$ such that $z_{i_0}\ne e$. Then $z_{i_0}>e$ because $z$ is a positive element in $\mathbb Z\, \overrightarrow{\mbox{\rm wr}}\, G$, and on the other hand, $z_{i_0}<e$ because $z\le a$ which is a contradiction. Consequently, $z_i = e$ for each $i>0$.


Now define an element $z''=(0,\langle z''_i: i \in \mathbb Z\rangle)\in \mathbb Z\, \overrightarrow{\mbox{\rm wr}}\, G$ in the following way: $z''_0=z_0,$ $z''_i= e$ for each $i>0$, $z''_{-1}= z_{-1}c$, and $z''_i= z_i$ for $i<-1$. Then $z''$ is a lower bound for $x$ and $y$, so that $z''\le z=x\wedge y$. But on the other hand, by construction of $z''$, we have $z<z''$ which is absurd. Hence, $G$ is linearly ordered.

The case (iii) $\Rightarrow$ (i) is dually to the previous implication.
\end{proof}

\begin{proposition}\label{pr:5.2}
Let $G$ be a directed po-group. If $\mathbb Z\, \overrightarrow{\mbox{\rm wr}}\, G$ $(\mathbb Z\, \overleftarrow{\mbox{\rm wr}}\, G)$ satisfies \RDP$_1$ (\RDP), then so satisfies $G$.
\end{proposition}

\begin{proof}
If $G$ is trivial, the statement of evident. So, we suppose that $G$ is non-trivial.

Assume the right wreath product satisfies RDP$_1$ (the case of RDP is analogous) and
let in $G$, we have $a_1a_2=b_1b_2$ for $a_1,a_2,b_1,b_2 \in G^+$. Without loss of generality we can assume that $a_1,a_2,b_1,b_2 >e$. We define the following elements in $\mathbb Z\, \overrightarrow{\mbox{\rm wr}}\, G$: $x_j=(0,\langle x_i^j: j \in \mathbb Z\rangle)$ and $y_j=(0,\langle y_i^j: j \in \mathbb Z\rangle)$ for $j=1,2$, where $x^j_0=a_j$, $y^j_0=b_j$ and $x^j_i=e=y^j_i$ for $i\ne 0$ and $j=1,2$. Then we have the following RDP$_1$ table

$$
\begin{matrix}
(0,\langle x_i^1\rangle)  &\vline & (0,\langle c_i^{11}\rangle) & (0,\langle c_i^{12}\rangle )\\
(0,\langle x_i^2\rangle) &\vline & (0,\langle c_i^{21}\rangle) & (0,\langle c_i^{22}\rangle)\\
  \hline     &\vline      &(0,\langle y^1_i \rangle)& (0,\langle y^2_i \rangle)
\end{matrix}\ \ ,
$$
where $(0,\langle c^{12}_i\rangle)\, \mbox{\rm \bf com}\, (0,\langle c_i^{21}\rangle)$. We assert that for each $i>0$, $c_i^{jk}= e$, where $j,k=1,2$. Indeed, let $i_0$ be the largest index $i>0$ such that $c_{i_0}\ne e$. Then $c_{i_0}>e$ and from $x_{i_0}^1=e=c_{i_0}^{11}c_{i_0}^{12}$ we conclude $c_{i_0}^{12}<e$. Hence, there is another integer $j_0>i_0$ such that $c_{j_0}^{12}>e$, otherwise $(0,\langle c_i^{12}\rangle)<0$. Hence, $c_{j_0}^{11}<e$ which is impossible. Hence, $c_i^{11}=e$ for each integer $i>0$. In the similar way we prove that all $c_i^{jk}=e$ for $i>0$ and $j,k=1,2$. In particular, we have the following RDP$_1$ table in $G$

$$
\begin{matrix}
a_1 &\vline & c_0^{11} & c_0^{12}\\
a_2 &\vline & c_0^{21} & c_0^{22}\\
  \hline     &\vline      &b_1 & b_2
\end{matrix}\ \
$$
and $c_0^{12}\, \mbox{\rm \bf com}\, c_0^{21}$.
\end{proof}

Now we present some results concerning the right and left wreath products and RDP's.

For any element $x=(n,\langle x_i\rangle)\in \mathbb Z\, \overrightarrow{\mbox{\rm wr}}\, G$ we denote by $supp(x):=\{i \in \mathbb Z: x_i\ne e\}$.

\begin{proposition}\label{pr:5.3}
Let $G$ be a directed po-group. Then the following statements are equivalent:

\begin{enumerate}
\item[{\rm (i)}] $G$ satisfies \RIP.
\item[{\rm (ii)}] $\mathbb Z\, \overrightarrow{\mbox{\rm wr}}\, G$ satisfies \RIP.
\item[{\rm (iii)}] $\mathbb Z\, \overleftarrow{\mbox{\rm wr}}\, G$ satisfies \RIP.
\end{enumerate}

\end{proposition}

\begin{proof}
For $G$ the statement is evident, so let us assume $G$ is non-trivial.

(ii) $\Rightarrow$ (i). Let $a_1,a_2\le b_1,b_2$ hold in $G$. To show RIP in $G$, it is enough to assume that $a_1\ne a_2$ and $b_1\ne b_2$. As in the proof of Proposition \ref{pr:5.2}, we define four elements $x_1=(0,\langle x_i^1: i \in \mathbb Z\rangle)$, $x_2=(0,\langle x_i^2: i \in \mathbb Z\rangle)$, $y_1=(0,\langle y_i^1: i \in \mathbb Z\rangle)$, and $y_2=(0,\langle y_i^2: i \in \mathbb Z\rangle)$. Then $x_1,x_2\le y_1,y_2$ in the right wreath product. Hence, there is an element $z=(0,\langle z_i: i \in \mathbb Z\rangle) \in \mathbb Z\, \overrightarrow{\mbox{\rm wr}}\, G$ such that $x_1,x_2 \le z\le y_1,y_2$. We assert that $z_i=e$ for any $i>0$. Indeed, let  $i_0=\max\{i >0: z_i\ne e\}$. Then $e<z_{i_0}<e$ which is impossible.  Now let $i_0(x_j)=\max\{i\in \mathbb Z: x_i^j\ne z_i\}$  and $i_0(y_j)=\max\{i\in \mathbb Z: y^j_i\ne z_i\}$. Then $x^j_{i_0(x_j)}< z_{i_0(x_j)}$ and $z_{i_0(y_j)}< y_{i_0(x_j)}^j$ for $j=1,2$.
If $i_0(x_j)=0$, then $a_j<z_0$. If $i_0(x_j)<0$, then $a_j=z_0$, so that $a_1,a_2 \le z_0$. Dually we show that $z_0\le b_1,b_2$.

In the analogous dual way we prove that (iii) implies (i).

(i) $\Rightarrow$ (ii). Let for $a_1=(n_1,\langle a_i^1: i \in \mathbb Z\rangle),$ $a_2=(n_2,\langle a_i^2: i \in \mathbb Z\rangle)$, $b_1= (m_1,\langle b_i^1: i \in \mathbb Z\rangle)$, and $b_2= (m_2,\langle b_i^2: i \in \mathbb Z\rangle)$ we have $a_1,a_2\le b_1,b_2$. Without loss of generality, we can assume that $n_1\le n_2\le m_1\le m_2.$ We have the following cases: (a) $n_2<m_1$. If $n_1 < n_2$, then $(n_1,\langle a_i^1: i \in \mathbb Z\rangle), (n_2,\langle a_i^2: i \in \mathbb Z\rangle)\le (n_2,\langle a_i^2: i \in \mathbb Z\rangle)\le (m_1,\langle b_i^1: i \in \mathbb Z\rangle), (m_2,\langle b_i^2: i \in \mathbb Z\rangle)$. If $n_1=n_2$, directness of $G$ implies that, for each $i \in supp(a_1)\cup supp(a_2)$, there is $d_i\in G$ such that $a^1_i,a^2_i < d_i$. We define $c=(n_1,\langle c_i: i \in \mathbb Z\rangle)$, where $c_i=d_i$ whenever $i \in supp(a_1)\cup supp(a_2)$ and $c_i = e$ otherwise. Then $a_1,a_2\le c\le b_1,b_2$. (b) $n_2=m_1$. If $n_1<n_2$, then $a_1<a_2$, so that $a_1,a_2\le a_2\le b_1,b_2$. Now let $n_1=n_2$. If $m_1 <m_2$, then $b_1<b_2$ so that $a_1,a_2\le b_1\le b_1,b_2$. Finally, let $m_1=m_2$, i.e. $n:=n_1=n_2=m_1=m_2$. If four, or three or two elements from $\{a_1,a_2,b_1,b_2\}$ coincide, the RIP holds trivially. Hence, we assume that $a_1,a_2 < b_1,b_2$, and in addition, $a_1 \ne a_2$ and $b_1\ne b_2$.

We define $i_0(j,k):=\max\{i: a_i^j\ne b^k_i\}$ for $j,k=1,2$; then $a^j_{i_0(j,k)}< b^k_{i_0(j,k)}$. Assume $i_0(1,1)\le i_0(2,1)$. Then $b^1_{i_0(1,1)} > a^1_{i_0(1,1)}$ and $b_i^1=a^1_i$ if $ i >i_0(1,1)$, and $b^1_{i_0(2,1)} >a^2_{i_0(2,1)}$ and  $b_i^1=a_i^2=a_i^1$ for $i> i_0(2,1)$.  Similarly, let $i_0(1,2)\le i_0(2,2)$. Then $b^2_{i_0(1,2)}> a^1_{i_0(1,2)}$ and $b_i^2=a^1_i$ if $i> i_0(1,2)$, and $b^2_{i_0(2,2)}> a^2_{i_0(2,2)}$ and $b_i^2=a_i^2=a_i^1$ for $i> i_0(2,2)$. It is sufficient to exhibit the following 18 cases:

(1) Let $i_0:=i_0(1,2)=i_0(2,2)=i_0(1,1)=i_0(2,1)$. Then $a^1_{i_0},a^2_{i_0}<b^1_{i_0}, b^2_{i_0}$. By RIP holding in $G$, there is an element $c'\in G$ such that $a^1_{i_0},a^2_{i_0}\le c'\le b^1_{i_0}, b^2_{i_0}$. If $a^1_{i_0},a^2_{i_0}< c'< b^1_{i_0}, b^2_{i_0}$, then the element $c=(n,\langle c_i: i\in \mathbb Z  \rangle)$, where $c_i=e$ if $i< i_0$, $c_{i_0}=c'$, and $c_i = a^1_i$ ($a_i^1=a_i^2=b_i^1=b_i^2$) for $i>i_0$. Then $a_1,a_2<c<b_1,b_2$. If $c'$ coincides with one of $\{a^1_{i_0},a^2_{i_0}, b^1_{i_0}, b^2_{i_0}\}$, say $a^1_{i_0}$, then we have two subcases: (i) $a^2_{i_0}<a^1_{i_0}$ and (ii) $a^2_{i_0}=a^1_{i_0}$. In the subcase (i) the element $c=a_1$, we have $a_1,a_2\le c \le b_1,b_2$. In the subcase (ii), there is $d\in G$ such that $d>a^1_{i_0-1}, a^2_{i_0-1}$. If we define $c_i=e$ if $i<i_0-1$, $c_{i_0-1}=d$ and $c_i=a^1_i$ if $i>i_0$. Then the element $c=(n,\langle c_i: i \in \mathbb Z\rangle)$ satisfies $a_1,a_2\le c \le b_1,b_2$. In a similar way we proceed with the case when $c'$ coincides with $b^1_{i_0}$ or $b^2_{i_0}$.

(2) Let $i_0(1,2)<i_0(2,2)\le i_0(1,1)< i_0(2,1)$. Then for $i_0=i_0(2,1)$, we have  $a^1_{i_0}=b^2_{i_0}=a^2_{i_0}$, $a^1_{i_0}=b^1_{i_0}$,  and $a^2_{i_0}< b^1_{i_0}$ which implies $b^2_{i_0}<b^1_{i_0}$ and $b^2_{i_0}=b^1_{i_0}$,  an absurd.

The same is true if we assume $i_0(1,2),i_0(2,2), i_0(1,1)< i_0(2,1)$ for arbitrary $i_0(1,2),i_0(2,2), i_0(1,1)$.

(3) Let $i_0(1,2)\le i_0(2,2)< i_0(1,1)= i_0(2,1)$. Then for $i_0=i_0(2,1)$, we have $a^1_{i_0}= b^2_{i_0}$, $a^2_{i_0}= b^2_{i_0}$, $a^1_{i_0} < b^1_{i_0}$, and $a^2_{i_0}< b^1_{i_0}$, which yields $a^1_{i_0}=b^2_{i_0} =a^2_{i_0}$ and $a^1_{i_0}<b^1_{i_0}$, $b^2_{i_0}<b^1_{i_0}$. For $i>i_0$, we have $a^1_i=a^2_i=b^1_i=b^2_i$. Hence, the element $c=b_2$ is an in-between element.

(4) Let $i_0(1,2)<i_0(2,2)= i_0(1,1)= i_0(2,1)$. Then for $i_0=i_0(2,2)$, we have $a^1_{i_0}=b^2_{i_0}$, $a^2_{i_0}<b^2_{i_0}$, $a^1_{i_0}<b^1_{i_0}$, and $a^2_{i_0}<b^1_{i_0}$. Then $b^2_{i_0}<b^1_{i_0}$ and the element $c=b_2$ is an in-between element.

(5) Let $i_0(1,2)=i_0(2,2)< i_0(1,1)< i_0(2,1)$. Then for $i_0=i_0(2,1)$, we have $a^1_{i_0}=b^2_{i_0}$, $a^2_{i_0}=b^2_{i_0}$, $a^1_{i_0}=b^1_{i_0}$, and $a^2_{i_0}<b^1_{i_0}$, which yields $b^2_{i_0}= b^1_{i_0}$ and $b^2_{i_0}< b^1_{i_0}$, an absurd; see also the end of (2).


(6) Let $i_0(1,2)=i_0(2,2)= i_0(1,1)< i_0(2,1)$. Then for $i_0=i_0(2,1)$, we have $a^1_{i_0}=b^2_{i_0}$, $a^2_{i_0}=b^2_{i_0}$, $a^1_{i_0}=b^1_{i_0}$, and $a^2_{i_0}<b^1_{i_0}$, which yields $b^2_{i_0}= b^1_{i_0}$ and $b^2_{i_0}< b^1_{i_0}$, an absurd; see also the end of (2).

Now let us assume that $i_0(2,2)$ is in the interval $[i_0(1,1), i_0(2,1)]$.

(7) Let $i_0(1,2)<i_0(1,1)\le i_0(2,2)< i_0(2,1)$. Then for $i_0=i_0(2,1)$, we have $a^1_{i_0}=b^2_{i_0}$, $a^1_{i_0} =b^1_{i_0}$, $a^2_{i_0}=b^2_{i_0}$ and  $a^2_{i_0}< b^1_{i_0}$ which implies $a^2_{i_0}=a^2_{i_0}$ and $a^2_{i_0}<a^1_{i_0}$, an absurd; see also the end of (2).

(8) Let $i_0(1,2)< i_0(1,1)< i_0(2,2)= i_0(2,1)$. Then for $i_0=i_0(2,1)$, we have $a^1_{i_0}=b^2_{i_0}$, $a^1_{i_0} =b^1_{i_0}$, $a^2_{i_0}<b^2_{i_0}$ and  $a^2_{i_0}< b^1_{i_0}$ which implies $a^2_{i_0}< a^1_{i_0}= b^1_{i_0}=b^2_{i_0}$. Hence, the element $c= a_1$ is an in-between element.

(9) Let $i_0(1,2)<i_0(1,1)= i_0(2,2)= i_0(2,1)$. Then for $i_0=i_0(2,1)$, we have $a^1_{i_0}=b^2_{i_0}$, $a^1_{i_0} <b^1_{i_0}$, $a^2_{i_0}<b^2_{i_0}$ and  $a^2_{i_0}< b^1_{i_0}$ which implies $a^2_{i_0}< a^1_{i_0}= b^2_{i_0}<b^1_{i_0}$. Hence, the element $c= a_1$ is an in-between element.

Now let us assume that $i_0(1,2)$ and $i_0(2,2)$ are in the interval $[i_0(1,1), i_0(2,1)]$.

(10) Let $i_0(1,1)\le i_0(1,2)< i_0(2,2)= i_0(2,1)$. Then for $i_0=i_0(2,1)$, we have $a^1_{i_0}=b^1_{i_0}$, $a^1_{i_0} =b^2_{i_0}$, $a^2_{i_0}<b^2_{i_0}$ and  $a^2_{i_0}< b^1_{i_0}$ which implies $a^2_{i_0}< a^1_{i_0}= b^1_{i_0}= b^2_{i_0}$. Hence, the element $c= a_1$ is an in-between element.

(11) Let $i_0(1,1)< i_0(1,2)= i_0(2,2)= i_0(2,1)$. Then for $i_0=i_0(2,1)$, we have $a^1_{i_0}=b^1_{i_0}$, $a^1_{i_0} <b^2_{i_0}$, $a^2_{i_0}<b^2_{i_0}$ and  $a^2_{i_0}< b^1_{i_0}$ which implies $a^2_{i_0}< a^1_{i_0}= b^1_{i_0}$. Hence, the element $c= a_1$ is an in-between element.

Finally, we assume $i_0(2,2)\le i_0(1,2)$,

(12) Let $i_0(2,2)\le i_0(1,2) <i_0(1,1)=i_0(2,1)$. Then for $i_0=i_0(2,1)$, we have $a^2_{i_0}=b^2_{i_0}$, $a^1_{i_0}=b^2_{i_0}$, $a^1_{i_0}<b^1_{i_0}$ and $a^2_{i_0}<b^1_{i_0}$, which implies $a^1_{i_0}=a^2_{i_0}=b^2_{i_0}< b^1_{i_0}$. Hence, the element $c=b_2$ is an in-between element.

(13) Let $i_0(2,2)< i_0(1,2) =i_0(1,1)=i_0(2,1)$. Then for $i_0=i_0(2,1)$, we have $a^2_{i_0}=b^2_{i_0}$, $a^1_{i_0}<b^2_{i_0}$, $a^1_{i_0}<b^1_{i_0}$ and $a^2_{i_0}<b^1_{i_0}$, which implies $a^1_{i_0}<b^1_{i_0},b^2_{i_0}$ and $a^2_{i_0}= b^2_{i_0}< b^1_{i_0}$. Hence, the element $c=b_2$ is an in-between element.

(14) Let $i_0(2,2)\le i_0(1,2) < i_0(1,1)=i_0(2,1)$. Then for $i_0=i_0(2,1)$, we have $a^2_{i_0}=b^2_{i_0}$, $a^1_{i_0}=b^2_{i_0}$, $a^1_{i_0}<b^1_{i_0}$ and $a^2_{i_0}<b^1_{i_0}$, which implies $a^1_{i_0}=b^2_{i_0}<b^1_{i_0}$ and $a^2_{i_0}= b^2_{i_0}< b^1_{i_0}$. Hence, the element $c=b_2$ is an in-between element.

(15) Let $i_0(2,2)\le i_0(1,1) < i_0(1,2)=i_0(2,1)$. Then for $i_0=i_0(2,1)$, we have $a^2_{i_0}=b^2_{i_0}$, $a^1_{i_0}=b^1_{i_0}$, $a^1_{i_0}<b^2_{i_0}$ and $a^2_{i_0}<b^1_{i_0}$, which implies $a^1_{i_0}= b^1_{i_0} <b^2_{i_0}$ and $a^2_{i_0}= b^2_{i_0} <b^1_{i_0}$. Consequently, $b_1 <b_2$ and $b_2<b_1$ which is impossible.

(16) Let $i_0(2,2)<i_0(1,1) = i_0(1,2)=i_0(2,1)$. Then for $i_0=i_0(2,1)$, we have $a^2_{i_0}=b^2_{i_0}$, $a^1_{i_0}<b^1_{i_0}$, $a^1_{i_0}<b^2_{i_0}$ and $a^2_{i_0}<b^1_{i_0}$, which implies $a^1_{i_0}<b^1_{i_0}, b^2_{i_0}$ and $a^2_{i_0}= b^2_{i_0}< b^1_{i_0}$. Hence, the element $c=b_2$ is an in-between element.

(17) Let $i_0(1,1)\le i_0(2,2) < i_0(1,2)=i_0(2,1)$. Then for $i_0=i_0(2,1)$, we have $a^1_{i_0}=b^1_{i_0}$, $a^2_{i_0}=b^2_{i_0}$, $a^1_{i_0}<b^2_{i_0}$ and $a^2_{i_0}<b^1_{i_0}$, which implies $a^1_{i_0}=b^1_{i_0}< b^2_{i_0}$ and $a^2_{i_0}= b^2_{i_0}< b^1_{i_0}$ which is impossible.

(18) Let $i_0(1,1)< i_0(2,2) = i_0(1,2)=i_0(2,1)$. Then for $i_0=i_0(2,1)$, we have $a^1_{i_0}=b^1_{i_0}$, $a^2_{i_0}<b^2_{i_0}$, $a^1_{i_0}<b^2_{i_0}$ and $a^2_{i_0}<b^1_{i_0}$, which implies $a^1_{i_0}=b^1_{i_0}< b^2_{i_0}$ and $a^2_{i_0}< b^2_{i_0}, b^1_{i_0}$. Hence, the element $c=b_1$ is an in-between element.

Summarizing all cases (1)--(18), we see that RIP holds for the right wreath product. In a dual way, we prove the validity of RIP for the left wreath product.
\end{proof}

Now we present a strengthening of Proposition \ref{pr:5.1}.

\begin{proposition}\label{pr:5.4}
Let $G$ be a non-commutative directed po-group. The following statements are equivalent
\begin{enumerate}
\item[{\rm (i)}] $\mathbb Z\, \overrightarrow{\mbox{\rm wr}}\, G$ satisfies \RDP$_1$.
\item[{\rm (ii)}] $\mathbb Z\, \overleftarrow{\mbox{\rm wr}}\, G$ satisfies \RDP$_1$.
\item[{\rm (iii)}] $G$ is a linearly ordered group.
\end{enumerate}
In  any such case, the right and left wreath products are linearly ordered groups.
\end{proposition}

\begin{proof}

(i) $\Rightarrow$ (ii), (iii). If $G$ is a trivial group, the statement is trivial. Thus we assume that $G$ is non-trivial. Suppose the converse of (iii), i.e., let $G$ be non-linearly ordered. Due to directness of $G$, there are two non-comparable elements $a$ and $b$ in $G$ such that there are $c,d\in G$ with $0<c<a,b< d$. We define two additional non-comparable elements $a'=a^{-1}d$ and $b'=b^{-1}d$, and four positive elements in the right wreath product $a_j= (0,\langle a^j_i: i \in \mathbb Z \rangle)$, $b_j= (0,\langle b^j_i: i \in \mathbb Z \rangle)$, $j=1,2$, where $a^j_i = e=b^j_i$ for $j\ne 0$, and $a^1_0=a$, $a^2_0=a'$, $b^1_0=b$, $b^2_0=b'$. Then $a_1a_2=b_1b_2$, and by the assumption, there are four positive elements $c_{jk}=(0,\langle c_i^{jk}: i \in \mathbb Z\rangle)$, $j,k=1,2$, such that for them RDP$_1$ decomposition holds. In addition, $c_{12}\, \mbox{\rm \bf com}\, c_{21}$. We claim that $c^{jk}_i=e$ for each $i>0$ ($j,k=1,2$). Indeed, if, e.g. $i_0=i_0^{11}:=\max\{i>0: c^{11}_i\ne e\},$ then $c_{i_0}^{11}>e$ which yields that $c_{i_0}^{12}<e$. Then there is an integer $i_1>i_0$ such that $c_{i_1}^{12}>e$ which yields $c_{i_1}^{11}<e$, a contradiction.

We assert that $c_{12}$ and $c_{21}$ are strictly positive elements in the right wreath product, that is $c_0^{12}, c_0^{21}>e$ holds in $G$. Indeed, if $c^{12}_0=e$, then $c_0^{11}=a$ and $c_0^{21}\ge e$. Indeed, either $c_0^{21}= e$, or $c_0^{21}\ne e$. In the first case we get $a^1_0=a=b^1_0=b$ which is impossible. In the second one, $c_0^{21}\ne e$, so that $c_0^{21}> e$. In any rate, we have $b=ac^{21}_0$ implying $a\le b$, a contradiction, therefore $c^{12}_0>e$. In the same way we prove that also $c^{21}_0>e$.

Since $G$ is non-trivial, there are two elements $x,y >e$ in $G$ such that $xy\ne yx$. Define two positive elements $x'=(0,\langle x_i: \in \mathbb Z \rangle )$ and $y'=(0,\langle y_i: \in \mathbb Z \rangle )$ such that $x_i=e=y_i$ for $i\ne -2,-1,$ and $x_{-2}=x$, $x_{-1}=y$ and $y_{-2}=y$, $y_{-1}=x$. Then $(0,\langle e\rangle )\le x'\le c_{12}$ and $(0,\langle e\rangle )\le y'\le c_{21}$ which means $x'y'=y'x'$ which is impossible while $xy\ne yx$.

Hence, our assumption that RDP$_1$ holds in the right wreath product for non-linear $G$ was false, whence, $G$ has to be linearly ordered.

In the same but dual way we proceed for the left right wreath product.

(iii) $\Rightarrow$ (i). If $G$ is linearly ordered, then the right and left wreath product is an $\ell$-group, so that it satisfies RDP$_2$ (see Proposition \ref{pr:5.1}), and finally, RDP$_1$ holds, too.
\end{proof}

{\bf Problem.} If a directed po-group $G$ satisfies RDP, does also the right (left) wreath product satisfy it?

A po-group $G$ is said to be {\it non-atomistic} if, for any $g>e$, there is a $g'>e$ such that $g>g'>e$.

\begin{proposition}\label{pr:5.5}
Let $G$ be a directed non-atomistic Abelian po-group. Then the following statements are equivalent \begin{enumerate}
\item[{\rm (i)}] $\mathbb Z\, \overrightarrow{\mbox{\rm wr}}\, G$ satisfies \RDP.
\item[{\rm (ii)}] $\mathbb Z\, \overleftarrow{\mbox{\rm wr}}\, G$ satisfies \RDP.
\item[{\rm (iii)}] $G$ satisfies \RDP.
\end{enumerate}
\end{proposition}

\begin{proof}
Again it is enough to assume that $G$ is non-trivial.

(i) $\Rightarrow$ (iii). It follows from Proposition \ref{pr:5.2}.

(iii) $\Rightarrow$ (i).
To exhibit RDP, it is enough to assume that in the equality $xy=uv$ all elements are strictly positive.

As usually, we will write $(n,\langle g_i\rangle)$ instead of $(n,\langle g_i: i \in \mathbb Z\rangle)$.

The positive cone of $\mathbb Z\, \overrightarrow{\mbox{wr}}\, G$ is the set of the following elements $\{(0,\langle x_i\rangle): x_j>e$ where $j$ is the greatest index $i$ such that $g_i\ne e\}\cup \{(n,\langle x_i\rangle): n>0\}$.

For $x=(n,\langle x_i\rangle)$ we denote by $supp(x):=\{i \in \mathbb Z: x_i\ne e\}$, and $i_0(x):=\max\{i: i \in supp(x)\}$.

We have the following cases:

(1) $(0, \langle  x_i \rangle) *(0,\langle  y_i\rangle)=
(0,\langle  u_i \rangle) *(0,\langle  v_i \rangle)$ for $x_i,y_i,u_i,v_i \in G$ for each $i \in I$.

Assume $x=(0,\langle x_i\rangle)$, $y=(0,\langle y_i\rangle)$, $u=(0,\langle u_i\rangle)$ and $v=(0,\langle x_i\rangle)$, and we can assume that $x,y,u,v$ are strictly positive.

Then $x_iy_i=u_iv_i$ for each $i \in \mathbb Z$ and $\max\{i_0(x),i_0(y)\}=\max\{i_0(u),i_0(v)\}$. If $i_0:=i_0(x)=i_0(y)=i_0(u)=i_0(v)$, then using the RDP of $G$, we have the following decomposition table

$$
\begin{matrix}
x_{i_0} &\vline & c_{i_0}^{11} & c_{i_0}^{12}\\
y_{i_0} &\vline & c_{i_0}^{21} & c_{i_0}^{22}\\
  \hline     &\vline      &u_{i_0} & v_{i_0}
\end{matrix}\ \ .
$$
where all $c_{i_0}^{jk}\ge e$, $j,k=1,2$.


%

We assert that there are positive elements $c_{jk}=(0,\langle c'^{jk}_i\rangle)$ for $j,k=1,2$ that form an RDP decomposition. Indeed, if $i<i_1$ or $i>i_0$, we set $c'^{jk}_i=e$. Let $i=i_0$. If all $c^{jk}_{i_0}>e$, we are ready and we put $c'^{jk}_{i_0}= c^{jk}_{i_0}$. Of course, at least two $c^{jk}_{i_0}$'s have to strictly positive. Now assume that, say $c^{12}_{i_0}=e$. Then $c^{11}_{i_0}>e$ and $c^{22}_{i_0} >e$. Then $c_{11}$ and $c_{22}$ can be chosen to be positive. If also $c^{21}_{i_0}>e$, then $c_{21}$ can be chosen to be positive. If $c^{21}_{i_0}=e$, due to directness of $G$, there is $d \in G$ such that $d> c^{jk}_{i_0-1}$ for $j,k=1,2$. Then

$$
\begin{matrix}
x_{i_0-1} &\vline & c_{i_0-1}^{11} & c_{i_0-1}^{12}\\
y_{i_0-1}&\vline & c_{i_0-1}^{21} & c_{i_0-1}^{22}\\
  \hline     &\vline      &u_{i_0-1} & v_{i_0-1}
\end{matrix}\ \
\quad =: \quad \begin{matrix}
x_{i_0-1} &\vline & d^{-1}c_{i_0-1}^{11} & dc_{i_0-1}^{12}\\
y_{i_0-1}&\vline & dc_{i_0-1}^{21} & d^{-1}c_{i_0-1}^{22}\\
  \hline     &\vline      &u_{i_0-1} & v_{i_0-1}
\end{matrix}\ \  (A).
$$
Here the elements in the second table have the property $dc_{i_0-1}^{12}>e$ and $dc_{i_0-1}^{21}>e$. From this table we can determine $c_{12}>e$ and $c_{21}>e$.

Now assume that e.g. $i_0(x)<i_0(y):=i_0$. Then we have the following simple decomposition table

$$
\begin{matrix}
x_{i_0}&\vline & e & e\\
y_{i_0} &\vline & c_{i_0}^{21} & c_{i_0}^{22}\\
  \hline     &\vline      &u_{i_0}& v_{i_0}
\end{matrix}\ \ .
$$

If $i_0(u)=i_0(v)$, then $c_{i_0}^{21}, c_{i_0}^{21}>e$, and if $i_0(u)<i_0(v)$, then $c_{i_0}^{21}=e$ and $c_{i_0}^{22}= v_{i_0}^{21}>e$.

In the first case, for each $i$ such that $i_0(x)<i<i_0(y)$, we have the decomposition $c_i^{11}=c_i^{12}=e$ and $c_i^{21}=u_i$, $c_i^{22}=v_i$. For $i=i_0(x)$ using directness of $G$, we have the following table

$$
\begin{matrix}
x_i &\vline & d_ic_i^{11} & c_i^{12}\\
y_i&\vline & c_i^{21} & c_i^{22}d_i\\
  \hline     &\vline      &u_i & v_i
\end{matrix}\ \ .
$$

If $c_i^{12}=e$, then $d_ic_i^{11}>e$ and using non-atomicity of $G$, we can assume without loss of generality $c_i^{12}>e$ and $d_ic_i^{11}>e$. If $c_i^{12}>e$, then

$$
\begin{matrix}
x_i &\vline & d_ic_i^{11}c_i^{12} & e\\
y_i&\vline & (c_i^{12})^{-1}c_i^{21} & c_i^{12}c_i^{22}d_i\\
  \hline     &\vline      &u_i & v_i
\end{matrix}\ \ (B)
$$
and using again non-atomicity of $G$, we can assume that the elements in the first row are strictly positive. Hence, we can already construct elements $c_{jk}$ in the right wreath product to be strictly positive, see (A).

Now assume the second case, i.e. $i_0(u)<i_0(v)$. For $i_0=i_0(v)$, we have the decomposition $c_{i_0}^{11}=c_{i_0}^{12}=c_{i_0}^{21}=e$ and $c_{i_0}^{22} =v_{i_0}>e$. For $i$ such that $\max\{i_0(x),i_0(u)\}<i<i_0$, we have the table $c_i^{11}=c_i^{12}=c_i^{21}=e$ and $c_i^{22} =v_i$. If $i=i_0(x)=i_0(u)$, using directness of $G$, we have the decomposition

$$
\begin{matrix}
x_i &\vline & d_ic_i^{11} & c_i^{12}\\
y_i&\vline & c_i^{21} & c_i^{22}d_i\\
  \hline     &\vline      &u_i & v_i
\end{matrix}\ \ .
$$

Using non-atomicity of $G$, if necessary, we can assume that the all elements in the left column are strictly positive.

If e.g. $i_0(x)<i_0(u)$, we have, for $i=i_0(u)$, a decomposition $c_i^{11}=c_i^{12}=e$ and $c_i^{21}=u_i>0$, $c_i^{22}=v_i$. For $i$ such that $i_0(x)<i<i_0(u)$, we have decomposition $c_i^{11}=c_i^{21}=e$, $c_i^{21}=u_i$, and $c_i^{22}=v_i$. For $i=i_0(x)$, we have the table

$$
\begin{matrix}
x_i &\vline & d_ic_i^{11} & c_i^{12}\\
y_i&\vline & c_i^{21} & c_i^{22}d_i\\
  \hline     &\vline      &u_i & v_i
\end{matrix}\ \ .
$$
If $c_i^{12}=e$, then $d_ic_i^{11}>e$ and using non-atomicity, we can assume that both elements in the first row are strictly positive. If $c_i^{12}>e$, we proceed as in table (B). Hence, we can construct strictly positive elements $c_{jk}$ in the right wreath product, compare with (A).

All other possibilities in case (1) can be proved analogously. Hence, (1) is proved completely.

(2) $(0,\langle x_i\rangle)* (n,\langle y_i\rangle)=(0,\langle y_i\rangle)* (n,\langle v_i\rangle)$, $n >0$. Then $x_{i_0(x)}>e$, $u_{i_0(y)}>e$, and $x_iy_i=u_iv_i$ for $i \in \mathbb Z$. Let $i_0=\max\{i: i_0(x),i_0(y),i_0(u),i_0(v)\}$. If $i_0(x)<i_0(u)$, then for $i=i_0(u)$, we have the decomposition $c_i^{11}=c_i^{12}=e$, $c_i^{21}=u_i>e$ and $c_i^{22}=v_i$. For $i>i_0(u)$ or for $i_0(x)<i<i_0(u)$, we have the decomposition $c_i^{11}=c_i^{12}=e$, $c_i^{21}=u_i$ and $c_i^{22}=v_i$. For $i=i_0(x)$, we use directness of $G$ and the table

$$
\begin{matrix}
x_i &\vline & d_ic_i^{11} & c_i^{12}\\
y_i&\vline & c_i^{21} & c_i^{22}d_i\\
  \hline     &\vline      &u_i & v_i
\end{matrix}\ \ (C).
$$
If $c_i^{12}=e$, then $d_ic_i^{11}>e$ and using non-atomicity of $G$, without loss of generality, we can assume that all elements in the first row and the first column  of (C) are strictly positive. If $c_i^{12}>e$, we take  table

$$
\begin{matrix}
x_i &\vline & d_ic_i^{11}c_i^{12} & e\\
y_i&\vline & (c_i^{12})^{-1}c_i^{21} & c_i^{12}c_i^{22}d_i\\
  \hline     &\vline      &u_i & v_i
\end{matrix}\ \ (D) ,
$$
where $d_ic_i^{11}c_i^{12}>e$, and using non-atomicity of $G$, we can again assume that all elements of the first row and the first column of table (C) are strictly positive.

The case $i_0(x)>i_0(u)$ can be treated in a similar way.

If $i_0(x)=i_0(u)$, for $i>i_0(u)$, then we have the decomposition $c_i^{11}=c_i^{12}=c_i^{21}=e$ and $c_i^{22}= v_i$. For $i=i_0(u)$, we have $x_i,u_i>e$. Using directness of $G$, for every $i\in \mathbb Z$, there is $d_i\le y_i,v_i$. Then $x_iy_id_i^{-1}= u_iv_id_i^{-1}$, and we have the following RDP table

$$
\begin{matrix}
x_i &\vline & c_i^{11} & c_i^{12}\\
y_i&\vline & c_i^{21} & c_i^{22}d_i\\
  \hline     &\vline      &u_i & v_i
\end{matrix}\ \ (E).
$$
If $c_i^{21}=e$, then $c_i^{11}>e$ and using non-atomicity of $G$, we can assume that  all elements in the first raw and in the first column of (E) are strictly positive. The same is true if $c_i^{12}=e$. Now let $c_i^{12}>e$ and $c_i^{21}>e$. If $c_i^{11}>e$, then we have that all elements in the first raw and in the first column of (E) are strictly positive. If $c_i^{11}=e$, $(*)$, we check elements of RDP decomposition for $i=i_0(u)-1$. For it we have the table (C). It can be rewritten in the form

$$
\begin{matrix}
x_i &\vline & c & c^{-1}d_ic_i^{11}c_i^{12}\\
y_i&\vline & c^{-1}d_ic_i^{11}c_i^{21} & c(d_ic_i^{11})^{-1}c_i^{22}d_i\\
  \hline     &\vline      &u_i & v_i
\end{matrix}\ \ ,
$$
where $c$ is an arbitrary strictly positive element of $G$.

Now let $i_1$ be an integer such that, for each $i\le i_1$, $x_i,y_i,u_i,v_i=e$. Then for $i$ such that $i_1\le i<i_0(u)$ $[i_1\le i<i_0(u)-1$ in case $(*)$], we take table (C), and for $i<i_1$, we set $c_i^{jk}=e$, $j,k=1,2$.

Hence, for (2), we have the following RDP table

$$
\begin{matrix}
(0,\langle x_i\rangle) &\vline & (0,\langle c_i^{11}\rangle) & (0,\langle c_i^{12}\rangle)\\
(n,\langle y_i\rangle)&\vline & (0,\langle c_i^{21}\rangle) & (n,\langle c_i^{22}\rangle)\\
  \hline     &\vline      &(0,\langle u_i\rangle) & (n,\langle v_i\rangle)
\end{matrix}\ \ .
$$

(3) $(n,\langle x_i\rangle)* (0,\langle y_i\rangle)=(n,\langle u_i\rangle)* (0,\langle v_i\rangle)$, $n>0$. Then $y_{i_0(y)}>e$, $v_{i_0(v)}>e$, and $x_iy_{i+n}=u_iv_{i+n}$ for $i \in \mathbb Z$. To exhibit cases $i_0(y)<i_0(v)$, $i_0(y)<i_0(v)$, and $i_0(y)=i_0(v)$, we follow methods of the proof of (2), concentrating to elements of RDP table in the second row and the second column.

Hence, for (3), we have the following RDP table

$$
\begin{matrix}
x_i &\vline & d_ic_i^{11}c_i^{12} & e\\
y_i&\vline & c_i^{21}(c_i^{12})^{-1} & c_i^{22}c_i^{12}d_i\\
  \hline     &\vline      &u_i & v_{i+n}
\end{matrix}\ \ ,
$$

$$
\begin{matrix}
(n,\langle x_i\rangle) &\vline & (n,\langle c_i^{11}\rangle) & (0,\langle c_{i-n}^{12}\rangle)\\
(0,\langle y_i\rangle)&\vline & (0,\langle c_{i-n}^{21}\rangle) & (0,\langle c_{i-n}^{22}\rangle)\\
  \hline     &\vline      &(n,\langle u_i\rangle) & (0,\langle v_i\rangle)
\end{matrix}\ \ .
$$

(4) $(0,\langle x_i\rangle)* (n,\langle y_i\rangle)=(n,\langle u_i\rangle)* (0,\langle v_i\rangle)$, $n>0$. Then $x_{i_0(x)}>e$, $v_{i_0(v)}>e$, and $x_iy_i=u_iv_{i+n}$ for $i \in \mathbb Z$. Assume $i_0(x)<i_0(v)-n$. If $i>i_0(v)-n$, we put $c_i^{11}=c_i^{12}=c_i^{22}=e$ and $c_i^{21}=u_i$. If $i=i_0(v)-n$, then we put $c_i^{11}=c_i^{12}=e$, $c_i^{21}=u_i$, $c_i^{22}=v_{i+n}>e$.  If $i=i_0(x)$, then

$$
\begin{matrix}
x_i &\vline & x_i & e\\
y_i&\vline & x_i^{-1}u_i & v_{i+n}\\
  \hline     &\vline      &u_i & v_{i+n}
\end{matrix}\ \ ,
$$
but $x_i>e$, so by non-atomicity of $G$, we have for any $e<c<x_i$ a table

$$
\begin{matrix}
x_i &\vline & x_ic^{-1} & c\\
y_i&\vline & x_i^{-1}u_ic & c^{-1}v_{i+n}\\
  \hline     &\vline      &u_i & v_{i+n}
\end{matrix}\ \ (F),
$$
where the first row has strictly positive elements.

If $i$ satisfies $i_0(x)<i<i_0(v)-n$, we take $c_i^{11}=c_i^{12}=e$, $c_i^{21}=u_i$, $c_i^{22}=v_{i+n}$.

If $i_1\le i<i_0(x)$, where $i_1$ is an integer such that if $i<i_1$, then $x_i,y_i,u_i,v_i=e$, we use directness, and for $i<i_1$, we set $c_i^{jk}=e$, $j,k=1,2$.

Now let $i_0(x)=i_0(v)-n$. If $i>i_0(v)-n$, we put $c_i^{11}=c_i^{12}=c_i^{22}=e$ and $c_i^{21}=u_i$. If $i=i_0(v)-n$, by directness, for each $i \in \mathbb Z$, there is $d_i\le x_i,y_i,u_i,v_{i+n}$. Hence, we have a table

$$
\begin{matrix}
x_i &\vline & d_ic_i^{11}c_i^{12} & e\\
y_i&\vline & c_i^{21}(c_i^{12})^{-1} & c_i^{22}c_i^{12}d_i\\
  \hline     &\vline      &u_i & v_{i+n}
\end{matrix}\ \ ,
$$
where $d_ic_i^{11}c_i^{12}$ and $c_i^{22}c_i^{12}d_i$ are strictly positive. For $i= i_0(v)-n$, we use directness, and we always can guarantee to be $c_i^{12}>e$. For the rest of this paragraph, we continue in the same way as in the previous one.

Now for the case $i_0(v)-n\le i_0(x)$, we proceed in the same way as for $i_0(v)-n\ge i_0(x)$. Finally, we have RDP table for (4)

$$
\begin{matrix}
(0,\langle x_i\rangle) &\vline & (0,\langle c_i^{11}\rangle) & (0,\langle c_{i-n}^{12}\rangle)\\
(n,\langle y_i\rangle)&\vline & (n,\langle c_{i-n}^{21}\rangle) & (0,\langle c_{i-n}^{22}\rangle)\\
  \hline     &\vline      &(n,\langle u_i\rangle) & (0,\langle v_i\rangle)
\end{matrix}\ \ .
$$

(5) $(n,\langle x_i\rangle)* (0,\langle y_i\rangle)=(m_1,\langle u_i\rangle)*(m_2,\langle v_i\rangle)$, where $m_1,m_2>0$, $n=m_1+m_2$. This yields $x_iy_{i+n}=u_iv_{i+m_1}$, and therefore, we have the following RDP table

$$
\begin{matrix}
(n,\langle x_i\rangle) &\vline & (m_1,\langle u_i\rangle) & (m_2,\langle u^{-1}_{i-m_1}x_{i-m_1}\rangle)\\
(0,\langle y_i\rangle)&\vline & (0,\langle e\rangle) & (0,\langle y_i\rangle)\\
  \hline     &\vline      &(m_1,\langle u_i\rangle) & (m_2,\langle v_i\rangle)
\end{matrix}\ \ .
$$

(6) $(0,\langle x_i\rangle)* (n,\langle y_i\rangle)=(m_1,\langle u_i\rangle)*(m_2,\langle v_i\rangle)$, where $m_1,m_2>0$, $n=m_1+m_2$. We rewrite it in the form $(m_1,\langle u_i\rangle)*(m_2,\langle v_i\rangle)=(0,\langle x_i\rangle)* (n,\langle y_i\rangle)$, and for it we have $u_iv_{i+m_1}=x_iy_i$. For it, we have RDP table

$$
\begin{matrix}
(m_1,\langle u_i\rangle) &\vline & (0,\langle x_i\rangle) & (m_1,\langle x^{-1}_iu_i\rangle)\\
(m_2,\langle v_i\rangle)&\vline & (0,\langle e\rangle) & (m_2,\langle v_i\rangle)\\
  \hline     &\vline      &(0,\langle x_i\rangle) & (n,\langle y_i\rangle)
\end{matrix}\ \ .
$$

(7) $(n_1,\langle x_i\rangle)* (n_2,\langle y_i\rangle)=(m_1,\langle u_i\rangle)*(m_2,\langle v_i\rangle)$, where $n_1,n_2,m_1,m_2>0$, $n_1+n_2=n=m_1+m_2$, and $m_1>n_1$. Then $x_iy_{i+n_1}=u_iv_{i+m_1}$, and we use the following RDP table

$$
\begin{matrix}
(m_1,\langle u_i\rangle) &\vline & (n_1,\langle x_i\rangle) & (m_1-n_1,\langle x^{-1}_{i-n_1}u_{i-n_1}\rangle)\\
(m_2,\langle v_i\rangle)&\vline & (0,\langle e\rangle) & (m_2,\langle v_i\rangle)\\
  \hline     &\vline      &(n_1,\langle x_i\rangle) & (n_2,\langle y_i\rangle)
\end{matrix}\ \  \mbox{ if } m_1>n_1.
$$

(8) $(n_1,\langle x_i\rangle)* (n_2,\langle y_i\rangle)=(m_1,\langle u_i\rangle)*(m_2,\langle v_i\rangle)$, where $n_1,n_2,m_1,m_2>0$, $n_1+n_2=n=m_1+m_2$, and $n_1>m_1$. Then $x_iy_{i+n_1}=u_iv_{i+m_1}$, and we use the following RDP table

$$
\begin{matrix}
(n_1,\langle x_i\rangle) &\vline & (m_1,\langle u_i\rangle) & (n_1-m_1,\langle u^{-1}_{i-m_1}x_{i-m_1}\rangle)\\
(n_2,\langle y_i\rangle)&\vline & (0,\langle e\rangle) & (n_2,\langle y_i\rangle)\\
  \hline     &\vline      &(m_1,\langle u_i\rangle) & (m_2,\langle v_i\rangle)
\end{matrix}\ \  \mbox{ if } n_1>m_1.
$$

(9) $(n_1,\langle x_i\rangle)* (n_2,\langle y_i\rangle)=(m_1,\langle u_i\rangle)*(m_2,\langle v_i\rangle)$, where $n_1,n_2,m_1,m_2>0$, $n_1+n_2=n=m_1+m_2$, and $n_1=m_1$. Then $x_iy_{i+n_1}=u_iv_{i+n_1}$.
Let $i_1$ be an integer such that $x_i,y_i,u_i,v_i=e$ if each $i<i_1$ and similarly, let $i_0$ be another integer such that $x_i,y_i,u_i,v_i=e$ for each $i>i_0$. For those $i$, we set $c_i^{11}, c_i^{12}, c_i^{21}, c_i^{22}=e$ and for $i$ with $i_1\le i\le i_0$, we apply directness of $G$, there is $d_i\le x_i,y_i,u_i,v_i$, so we have $d^{-1}_ix_iy_{i+n_1}d^{-1}_{i+n_1} = d_i^{-1}u_iv_{i+n_1}d_{i+n_1}^{-1}$, and we obtain table

$$
\begin{matrix}
x_i &\vline & d_ic_i^{11} & c_i^{12}\\
y_{i+n_1}&\vline & c_i^{21} & c_i^{22}d_{i+n_1}\\
  \hline     &\vline      &u_i & v_{i+n_1}
\end{matrix}\ \ .
$$
Without loss of generality, we can assume that $c^{12}_i>e$, if not, we take an arbitrary $c>e$ so that in the table

$$
\begin{matrix}
x_i &\vline & d_ic_i^{11}c^{-1} & c\\
y_{i+n_1}&\vline & c_i^{21}c & c^{-1}c_i^{22}d_{i+n_1}\\
  \hline     &\vline      &u_i & v_{i+n_1}
\end{matrix}\ \ ,
$$
the elements in the right-up-hand and left-down-hand are strictly positive.

Then we have an RDP table for (9)

$$
\begin{matrix}
(n_1,\langle x_i\rangle) &\vline & (n_1,\langle d_ic_i^{11}\rangle) & (0,\langle c_{i-n_1}^{12}\rangle)\\
(n_2,\langle y_i\rangle)&\vline & (0,\langle c_{i-n_1}^{21}\rangle) & (n_2,\langle c_{i-n_1}^{22}\rangle)\\
  \hline     &\vline      &(n_1,\langle u_i\rangle) & (n_2,\langle v_i\rangle)
\end{matrix}\ \ .
$$

Dually we prove the equivalence (ii) $\Leftrightarrow$ (iii).
\end{proof}

\section{Conclusion}

We have established conditions when the lexicographic product $G=G_1 \lex G_2$ has some kind of RDP's where $G_1$ is a po-group and $G_2$ is a directed po-group. Theorem \ref{th:3.1} gives an answer when $G$ has RDP, RDP$_1$ or RDP$_2$, under the condition $G_1$ is a linearly ordered group. Theorem \ref{th:3.3} gives an answer that $G$ has RDP if $G_1$ is an antilattice. We note that we do not know any relevant answer for a question when $G$ has RDP$_1$. Theorem \ref{th:4.1} answers when Wreath product $A\, \mbox {\rm Wr } G$ has RDP, RDP$_1$ and RDP$_2$.  In Section 5, we have established conditions when the right and left wreath products have RDP's.

Applications of these results are important for the representation of effect algebras, pseudo effect algebras, MV-algebras and pseudo MV-algebras by an interval in the lexicographic product or in the wreath product, or for kite $n$-perfect pseudo effect algebras, see \cite{BoDv}.

\end{document}